\documentclass[12pt,reqno]{amsart}

\usepackage[usenames,dvipsnames]{xcolor}
\RequirePackage[OT1]{fontenc}

\usepackage[colorlinks,citecolor=blue,urlcolor=blue,linkcolor=blue,linktocpage=true]{hyperref}

\usepackage{amsmath,amssymb,amsthm,amsfonts,amsbsy,latexsym,amsxtra}

\usepackage{graphicx}

\usepackage{todonotes}
\usepackage{appendix}
\usepackage{url}

\allowdisplaybreaks[4]

\usepackage[usenames,dvipsnames]{xcolor}            
\usepackage{graphicx}
\usepackage{bbm,bm}
\usepackage{comment}
\usepackage{mathtools}
\usepackage{placeins}
\usepackage[labelfont={sl}]{caption}
\usepackage[labelfont={normalfont}]{subcaption}

\usepackage{tikz}
\usetikzlibrary{decorations.pathreplacing}
\usetikzlibrary{calc}

\usepackage{amssymb,amsfonts,amsmath,amsthm,amscd,dsfont,mathrsfs}

\usepackage{url}

\usepackage{enumerate}
\usepackage{hyperref}

\numberwithin{equation}{section}

\newtheorem{theorem}{Theorem}[section]

\newtheorem{prop}[theorem]{Proposition}
\newtheorem{lemma}[theorem]{Lemma}
\newtheorem{corollary}[theorem]{Corollary}

\theoremstyle{definition}

\newcommand{\E}{\mathbf{E}}
\newcommand{\R}{\mathbb{R}}

\newcommand{\sM}{\mathcal{M}}

\newcommand{\N}{\mathbb{N}}

\renewcommand{\P}{\mathbf{P}}
\newcommand{\Prob}[1]{\mathbf P\left\{#1\right\}}

\renewcommand{\emptyset}{\varnothing}
\newcommand{\XX}{\mathbb{X}}
\newcommand{\QQ}{\mathbb{Q}}

\newcommand{\Nb}{\mathbf{N}}

\newcommand{\ind}{\mathds{1}}

\newcommand{\bx}{{\boldsymbol{x}}}
\newcommand{\by}{{\boldsymbol{y}}}
\newcommand{\bz}{{\boldsymbol{z}}}
\newcommand{\I}{\mathfrak I}
\newcommand{\spm }[1][d]{\nu_{#1}} 
\newcommand{\bvol}{\omega_d}
\newcommand{\qtwo }{M_{\alpha}'}

\newlength{\wdth}

\newcommand{\diff}{{\mathrm d}}

\renewcommand{\prec}{\preceq}
\renewcommand{\succ}{\succeq}

\DeclareMathOperator{\Var}{Var}

\usepackage{color}

\usepackage[numbers, sort&compress]{natbib}

\makeatletter
\@namedef{subjclassname@2020}{\textup{2020} Mathematics Subject Classification}
\makeatother

\begin{document}

\title[CLT for a birth-growth model with Poisson arrivals]{Central limit
  theorem for a birth-growth model with Poisson arrivals and random
  growth speed}

\author[C.Bhattacharjee]{Chinmoy Bhattacharjee}
\address{Department of Mathematics, University of Hamburg, Bundesstrasse 55, 20146 Hamburg, Germany}
\email{chinmoy.bhattacharjee@uni-hamburg.de}

\author[I.Molchanov]{Ilya Molchanov}
\address{IMSV, University of Bern, Alpeneggstrasse 22, 3012 Bern, Switzerland}
\email{ilya.molchanov@stat.unibe.ch}

\author[R.Turin]{Riccardo Turin}
\address{Swiss Re Management
	Ltd, Mythenquai 50/60, 8022 Zurich, Switzerland}
\email{Riccardo\_Turin@swissre.com}

\date{\today}



\begin{abstract}
  We consider Gaussian approximation in a variant of the classical
  Johnson--Mehl birth-growth model with random growth speed. Seeds
  appear randomly in $\mathbb{R}^d$ at random times and start growing
  instantaneously in all directions with a random speed. The location,
  birth time and growth speed of the seeds are given by a Poisson
  process.  Under suitable conditions on the random growth speed, the
  time distribution and a weight function
  $h:\mathbb{R}^d \times [0,\infty) \to [0,\infty)$, we prove a
  Gaussian convergence of the sum of the weights at the exposed
  points, which are those seeds in the model that are not covered at
  the time of their birth. Such models have previously been
  considered, albeit with fixed growth speed. Moreover, using recent
  results on stabilization regions, we provide non-asymptotic bounds
  on the distance between the normalized sum of weights and a standard
  Gaussian random variable in the Wasserstein and Kolmogorov metrics.
\end{abstract}

\maketitle

\section{Introduction}

In the spatial Johnson--Mehl growth model, seeds arrive at random
times $t_i$, $i \in \N$, at random locations $x_i$, $i \in \N$, in
$\R^d$, according to a Poisson process $(x_i,t_i)_{i \in \N}$ on
$\R^d \times \R_+$, where $\R_+:=[0,\infty)$. Once a seed is born at
time $t$, it begins to form a cell by growing radially in all
directions at a constant speed $v\geq 0$, so that by time $t'$ it
occupies the ball of radius $v(t'-t)$. The parts of the space claimed
by the seeds form the so-called Johnson--Mehl tessellation, see
\cite{cskm} and \cite{obsc}. This is a generalization of the classical
Voronoi tessellation, which is obtained if all births occur
simultaneously at time zero.

The study of such birth-growth processes started with the work of
Kolmogorov~\cite{Kol37} in two dimensions to model crystal
growth. Since then, this model has seen applications in various
contexts such as phase transition kinetics, polymers, ecological
systems and DNA replications to name a few, see \cite{BR08,cskm,obsc}
and references therein. A central limit theorem for the Johnson--Mehl
model with inhomogeneous arrivals of the seeds was obtained in
\cite{CQ97}.

Variants of the classical spatial birth-growth model can be found,
sometimes as a particular case of other models, in many papers. Among
them, we mention \cite{PY02} and \cite{BY05}, where the birth-growth
model appears as a particular case of a random sequential packing
model, and \cite{SY08}, who have studied a variant of the model with
non-uniform deterministic growth patterns.  The main tools rely on the
concept of stabilization by considering regions where the appearance
of new seeds influences the functional of interest.

In this paper, we consider a generalization of the Johnson--Mehl model
by introducing random growth speeds for the seeds. This gives rise to
many interesting features in the model, most importantly, long-range
interactions if the speed can take arbitrarily large values with
positive probability. Therefore, the model with random speed is no
longer stabilizing in the classical sense of \cite{pen:yuk03} and
\cite{LSY19}, since distant points may influence the growth pattern if
their speeds are sufficiently high. It should be noted that, even in
the constant speed setting, we substantially improve and extend limit
theorems obtained in \cite{CQ97}.

We consider a birth-growth model, determined by a Poisson process
$\eta$ in $\XX:=\R^d\times\R_+\times\R_+$ with intensity measure
$\mu:=\lambda\otimes\theta\otimes\nu$, where $\lambda$ is the Lebesgue
measure on $\R^d$, $\theta$ is a non-null locally finite measure on
$\R_+$, and $\nu$ is a probability distribution on $\R_+$ with
$\nu(\{0\})<1$.  Each point $\bx$ of this point process $\eta$ has
three components $(x,t_x,v_x)$, where $v_x \in \R_+$ denotes the
random speed of a seed born at location $x \in \R^d$ and whose growth
commences at time $t_x\in \R_+$. In a given point configuration, a
point $\bx:=(x,t_x,v_x)$ is said to be \textit{exposed} if there is no other
point $(y,t_y,v_y)$ in the configuration with $t_y <t_x$ and
$\|x-y\| \le v_y(t_x-t_y)$, where $\|\cdot\|$ denotes the Euclidean
norm. Notice that the event that a
point $(x,t_x,v_x)\in \eta$ is exposed depends only on the point
configuration in the region
\begin{equation}
  \label{eq:3}
  L_{x,t_x}:=\big\{(y,t_y,v_y)\in \XX: \|x-y\| \le v_y(t_x-t_y)\big\}.
\end{equation}
Namely, $\bx$ is exposed if and only if $\eta$ has no points (apart
from $\bx$) in $L_{x,t_x}$.

The \textit{growth frontier} of the model can be defined as the random field
\begin{equation}\label{eq:grfr}
  \min_{(x,t_x,v_x)\in\eta} \Big(t_x+\frac{1}{v_x}\|y-x\|\Big),
  \quad y\in\R^d.
\end{equation}
This is an example of an extremal shot-noise process, see
\cite{hein:mol94}.  Its value at a point $y \in \R^d$ corresponds to a
seed from $\eta$ whose growth region covers $y$ first. It should be
noted here that this covering seed need not be an exposed one. In
other words, because of random speeds, it may happen that the cell
grown from a non-exposed seed shades a subsequent seed which would be
exposed otherwise. This excludes possible applications of our model
with random growth speed to crystallisation, where a more natural
model would be to not allow a non-exposed seed to affect any future
seeds. But this creates a causal chain of influences that seems quite
difficult to study with the currently known methods of stabilization
for Gaussian approximation. 

Nonetheless, models such as ours are natural in telecommunication
applications, where the speed plays the role of the weight or strength
of a particular transmission node, where the growth frontier defined
above can be used as a variant of the additive signal-to-interference
model from \cite[Chapter~5]{baccelli10}. Furthermore, similar models
can be applied in the ecological or epidemiological context, where a
non-visible event influences appearances of others. Suppose we have a
barren land and a drone/machine is planting seeds from a mixture of
plant species at random times and random locations for
reforestation. Each seed, after falling on the ground, starts growing
a bush around it at a random speed depending on its species. If a new
seed falls on a part of the ground that is already covered in bushes,
it is still allowed to form its own bush, i.e., there is no
exclusion. Now the number of exposed points in our model above
translates to the number of seeds that start a bush on a then barren
land, rather than starting on a ground already covered in bushes. This
in some sense, can explain the efficiency of the reforestation
process, i.e., what fraction of the seeds were planted on barren land,
in contrast to planting them on an already existing bush.

Given a measurable weight function $h:\R^d \times \R_+ \to \R_+$ the
main object of interest in this paper is the sum of $h$ over the space-time coordinates $(x,t_x)$ of the 
exposed points in $\eta$. These can be defined as those points $(y,t_y)$ where the growth frontier defined at \eqref{eq:grfr} has a local minimum (See Section \ref{sec:model-main-results} for a precise definition). Our aim is to provide sufficient
conditions for Gaussian convergence of such sums.  A standard approach
for proving Gaussian convergence for such statistics relies on
stabilization theory \cite{BY05, ERS15, PY02, SY08}. While in the
stabilization literature, one commonly assumes that the so-called
stabilization region is a ball around a given reference point,
the region $L_{x,t_x}$ is unbounded and it seems that it is not
expressible as a ball around $\bx$ in some different metric.
Moreover, our stabilization region is set to be empty if $\bx$
is not exposed. 

The main challenge when working with random unbounded speed of growth
is that there are possibly very long-range interactions between
seeds. This makes the use of balls as stabilization regions vastly
suboptimal and necessitates the use of regions of a more general
shape. In particular, we only assume that the random growth speed in
our model has finite moment of order $7d$ (see assumption (C) in
Section~\ref{sec:model-main-results}), and this allows for some
power-tailed distributions for the speed.

The recent work \cite{bha:mol2021} introduced a new notion of
\emph{region-stabilization} which allows for more general regions than
balls, and, building on the seminal work \cite{LPS16}, provides bounds
on the rate of Gaussian convergence for certain sums of
region-stabilizing score functions. We will utilize this to derive
bounds on the Wasserstein and Kolmogorov distances, defined below,
between a suitably normalized sum of weights and the standard Gaussian
distribution. For real-valued random variables $X$ and $Y$, the
\emph{Wasserstein distance} between their distributions is given by
$$
d_{\mathrm{W}}(X,Y):= \sup_{f \in \operatorname{Lip}_1} |\E\; f(X) - \E \; f(Y)|,
$$
where $\operatorname{Lip}_1$ denotes the class of all Lipschitz
functions $f: \R \to \R$ with Lipschitz constant at most one. The
\emph{Kolmogorov distance} between the distributions is given by
$$
d_{\mathrm{K}}(X,Y):= \sup_{t \in \R} |\Prob{X \le t} - \Prob{Y \le t}|.
$$

The rest of the paper is organized as follows. In
Section~\ref{sec:model-main-results}, we describe the model and state
our main results. In Section~\ref{sec:variance-estimation}, we prove a
result providing necessary upper and lower bounds for the variance of
our statistic of interest.  Section~\ref{sec:proof-theor-refthm:b}
presents the proofs of our quantitative bounds.

\section{Model and main results}
\label{sec:model-main-results}

Recall that we work in the space $\XX:=\R^d\times\R_+\times\R_+$,
$d \in \N$, with the Borel $\sigma$-algebra. The points
from $\XX$ are written as $\bx:=(x,t_x,v_x)$, so that $\bx$ designates
a seed born in position $x$ at time $t_x$, which then grows radially
in all directions with speed $v_x$. For $\bx\in \XX$, the set
$$
G_{\bx}=G_{x,t_x,v_x}:=\big\{(y,t_y) \in \R^d \times \R_+ :
t_y \ge t_x, \|y-x\| \le v_x (t_y-t_x)\big\}
$$
is the growth region of the seed $\bx$. Denote by $\Nb$ the family of
$\sigma$-finite counting measures $\sM$ on $\XX$ equipped with the
smallest $\sigma$-algebra $\mathscr{N}$ such that the maps
$\sM \mapsto \sM(A)$ are measurable for all Borel $A$.  We write
$\bx\in\sM$ if $\sM(\{\bx\})\geq 1$. For $\sM \in \Nb$, a point
$\bx \in \sM$ is said to be \emph{exposed} in $\sM$ if it does not
belong to the growth region of any other point
$\by\in\sM$, $\by \neq \bx$. Note that the property of being exposed is not
influenced by the speed component of $\bx$.

The \emph{influence set} $L_\bx=L_{x,t_x}$, $\bx \in \XX$, defined at
\eqref{eq:3}, is exactly the set of points that were born before time
$t_x$ and which at time $t_x$ occupy a region that covers the location
$x$, thereby shading it. Note that $\by\in L_\bx$ if and only if
$\bx\in G_\by$. Clearly, a point
$\bx\in\sM$ is exposed in $\sM$ if and only if
$\sM (L_\bx\setminus\{x\})=0$.
We write $(y,t_y,v_y)\preceq (x,t_x)$ or $\by \preceq \bx$ if
$\by\in L_{x,t_x}$ (recall that the speed component of $\bx$ is
irrelevant in such a relation) and so $\bx$ is not an exposed point
with respect to $\delta_{\by}$, where $\delta_{\by}$ denotes the Dirac
measure at $\by$.

For $\sM \in \Nb$ and $\bx \in \sM$, denote
\begin{displaymath}
  H_\bx(\sM)\equiv H_{x,t_x}(\sM)
  :=\ind\{\bx \text{ is exposed in $\sM$}\}
  =\ind_{\sM(L_{x,t_x} \setminus \{\bx\})=0}.
\end{displaymath}
A generic way to construct an additive functional on the exposed
points is to consider the sum of weights of these points, where each
exposed point $\bx$ contributes a weight $h(\bx)$ for some measurable
$h: \XX \to \R_+$.  In the following we consider weight functions
$h(\bx)$ which are products of two measurable functions
$h_1:\R^d \to \R_+$ and $h_2: \R_+ \to \R_+$ of the locations and
birth times, respectively, of the exposed points.  In particular, we
let $h_1(x) = \ind_W(x)=\ind\{x \in W\}$ for a window
$W \subset \R^d$, and $h_2(t) = \ind\{t \le a\}$ for
$a \in (0,\infty)$. Then
\begin{equation}
  \label{eq:F}
  F(\sM)
  :=\int_\XX h_1(x) h_2(t_x)H_{\bx}(\sM)\sM(\diff \bx)
  =\sum_{\bx\in\sM} \ind_{x \in W}\ind_{t_x \le a} H_{\bx}(\sM)
\end{equation}
is the number of exposed points from $\sM$ located in $W$ and
born before time $a$. Note here that when we add a new point $\by=(y,t_y,v_y) \in \R^d \times \R_+ \times \R_+$ to a configuration $\sM \in \Nb$ not containing it, the change in the value of $F$ is not a function of only $\by$ and some local neighbourhood of it, but rather it depends on points in the configuration that might be very far away. Indeed, we have for $\by \notin \sM$,
$$
F(\sM + \delta_{\by}) - F(\sM)=  \ind_{y \in W}\ind_{t_y \le a} H_{\by}(\sM + \delta_{\by}) - \sum_{\bx\in\sM} \ind_{x \in W}\ind_{t_x \le a} \mathds{1}_{\bx \in L_{(y,t_y)}},
$$
that is, $F$ may increase by one when $\by$ is exposed in $\sM + \delta_{\by}$, while simultaneously, any point $\bx\in\sM$ which was previously exposed in $\sM$ may not be so anymore after adding $\by$, if it happens to fall in the influence set $L_{(y,t_y)}$ of $\by$. This necessitates the use of region-stabilization.

Recall that $\eta$ is a \emph{Poisson process} in $\XX$ with intensity
measure $\mu$, being the product of the Lebesgue measure $\lambda$ on
$\R^d$, a non-null locally finite measure $\theta$ on $\R_+$, and
a probability measure $\nu$ on $\R_+$ with $\nu(\{0\})<1$. Note
that $\eta$ is a simple random counting measure.  The main goal of
this paper is to find sufficient conditions for a Gaussian convergence
of $F\equiv F(\eta)$ as defined at \eqref{eq:F}.
The functional $F(\eta)$ is a region-stabilizing functional, in the
sense of \cite{bha:mol2021}, and can be represented as
$F(\eta)=\sum_{\bx \in \eta} \xi(\bx, \eta)$, where the score function
$\xi$ is given by
\begin{equation}
  \label{eq:4}
  \xi(\bx,\sM):=\ind_{x \in W}\ind_{t_x \le a}
  H_{\bx}(\sM), \quad \bx \in \sM,
\end{equation}
with the region of stabilization being $L_{x,t_x}$ when $\bx$ is an exposed point (See Section \ref{sec:proof-theor-refthm:b} for more details). As a convention, let $\xi(\bx,\sM)=0$ if $\sM=0$ or if
$x\notin\sM$. Theorem~2.1 in \cite{bha:mol2021} yields ready-to-use
bounds on the Wasserstein and Kolmogorov distances between $F$,
suitably normalized, and a standard Gaussian random variable $N$ upon
validating equation (2.1) and conditions (A1) and (A2) therein. We
consistently follow the notation of \cite{bha:mol2021}.

Now we are ready to state our main results.  First, we list the
necessary assumptions on our model.  In the sequel, we drop $\lambda$
in Lebesgue integrals and simply write $\diff x$ instead of
$\lambda(\diff x)$.
\begin{enumerate}[(A)]
\item 
The window $W$ is compact convex with nonempty interior.
\item For all $x>0$, 
  \begin{equation*}
    \int_0^\infty e^{-x \Lambda(t)} \, \theta(\diff t)<\infty,
  \end{equation*}
  where 
  \begin{equation}
    \label{eq:7}
    \Lambda(t):=\bvol\int_0^t (t-s)^d \theta(\diff s)
  \end{equation}
 and  $\bvol$ is the volume of the $d$-dimensional unit Euclidean ball.
  \smallbreak 
\item The moment of $\nu$ of order
  $7d$ is finite, i.e., $\spm[7d]<\infty$, where 
  \begin{displaymath}
    \spm[u]:=\int_0^\infty v^u \nu(\diff v),\quad u\geq 0.
  \end{displaymath}
\end{enumerate}
 
Note that the function $\Lambda(t)$ given at \eqref{eq:7} is, up to a
constant, the measure of the influence set of any point $\bx \in \XX$
with time component $t_x=t$ (the measure of the influence set does not
depend on the location and speed components of $\bx$). Indeed, the
$\mu$-content of $L_{x,t_x}$ is given by
\begin{align*}
  \mu(L_{x,t_x})
  &=\int_{0}^\infty \int_0^{t_x}
    \int_{\R^d} \ind_{y \in B_{v_y(t_x-t_y)}(x)}
    \,\diff y\,\theta(\diff t_y)\nu(\diff v_y)\\
  & =\int_0^\infty \nu(\diff v_y) \int_0^{t_x}
    \bvol v_y^d(t_x-t_y)^d
    \theta(\diff t_y)=\spm  \Lambda(t_x),
\end{align*}
where $B_r(x)$ denotes the closed
$d$-dimensional Euclidean ball of radius $r$ centered at $x \in
\R^d$. In particular, if $\theta$ is the Lebesgue measure on
$\R_+$, then $\Lambda(t)=\bvol\,t^{d+1}/(d+1)$.
 


The following theorem is our first main result. We denote by
$(V_j(W))_{0 \le j \le d }$ the intrinsic volumes of $W$ (see
\cite[Section~4.1]{schn2}), and let
\begin{equation}\label{eq:V}
  V(W):=\max_{0 \le j \le d} V_j(W).
\end{equation}

\begin{theorem}\label{thm:birth}
  Let $\eta$ be a Poisson process on $\XX$ with intensity measure
  $\mu$ as above, such that
  the assumptions (A)--(C) hold.
  Then, for $F:=F(\eta)$ as in \eqref{eq:F} with $a \in (0,\infty)$,
  \begin{displaymath}
    d_{\mathrm{W}}\left(\frac{F - \E F}{\sqrt{\Var
          F}},N\right) \le C \Bigg[\frac{\sqrt{V(W)}}{\Var F}
    +\frac{V(W)}{(\Var F)^{3/2}}\Bigg],
  \end{displaymath}
  and
  \begin{displaymath}
    d_{\mathrm{K}}\left(\frac{F - \E F}{\sqrt{\Var
          F}},N\right) \le
    C \Bigg[\frac{\sqrt{V(W)}}{\Var F}\\
    +\frac{V(W)}{(\Var F)^{3/2}}
    +\frac{V(W)^{5/4}
      + V(W)^{3/2}}{(\Var F)^{2}}\Bigg]
  \end{displaymath}
  for a constant $C \in (0,\infty)$ which depends on $a$, $d$, the first $7d$ moments of
  $\nu$, and $\theta$.
\end{theorem}

To derive a quantitative central limit theorem from
Theorem~\ref{thm:birth}, a lower bound on the variance is needed.  The
following proposition provides general lower and upper bounds on the
variance, which are then specialized for measures on
$\R_+$ given by 
\begin{equation}
  \label{eq:10}
  \theta(\diff t):=t^\tau \diff t, \quad \tau\in(-1,\infty).
\end{equation}
In the following, $t_1 \wedge t_2$ denotes
$\min\{t_1,t_2\}$ for $t_1, t_2 \in \R$.
For $a \in (0,\infty)$ and $\tau>-1$, define the function 
\begin{equation}\label{eq:la}
  l_{a,\tau}(x):=\gamma\left(\frac{\tau+1}{d+\tau+1},a^{d+\tau+1}x\right)
  x^{-(\tau+1)/(d+\tau+1)}, \quad x>0,
\end{equation}
where $\gamma(p,z):=\int_0^z t^{p-1}e^{-t}\diff t $ is the lower incomplete Gamma function.

\begin{prop}
  \label{prop:var}
  Let the assumptions (A)--(C) be in force.  For a Poisson process
  $\eta$ with intensity measure $\mu$ as above and $F:=F(\eta)$ as in
  \eqref{eq:F},
  \begin{equation}
    \label{eq:9}
    \frac{\Var(F)}{\lambda(W)} \ge 
    \Bigg[\int_{0}^a w(t) \theta(\diff t)
    - 2 \bvol \spm  \int_0^a \int_0^{t} (t-s)^d w(s) w(t)
    \theta(\diff s) \theta(\diff t)\Bigg]
  \end{equation}
  and
  \begin{multline}
  \label{eq:vub}
    \frac{\Var(F)}{\lambda(W)} \le 
    \Bigg[2 \int_{0}^a w(t)^{1/2} \theta(\diff t)\\
    + \bvol^2 \nu_{2d}  \int_{[0,a]^2} \int_{0}^{t_1 \wedge t_2}
    (t_1-s)^d (t_2-s)^d w(t_1)^{1/2} w(t_2)^{1/2}  \theta(\diff s)
    \theta^2(\diff(t_1,t_2))\Bigg],
  \end{multline}
  where
  \begin{equation}
    \label{eq:6}
    w(t):=e^{-\spm  \Lambda(t)}
    =\E\left[H_{0,t}(\eta)\right].
  \end{equation}
  If $\theta$ is given by
  \eqref{eq:10}, then
  \begin{equation}\label{eq:vgam}
    C_1  (d-1-\tau) <C_1'\leq
    \frac{\Var(F)}{\lambda(W) l_{a,\tau}(\nu_d)}\leq
    C_2 (1+\nu_{2d} \nu_d^{-2})
  \end{equation}
  for constants $C_1, C_1', C_2$ depending on the dimension $d$ and
  $\tau$, and $C_1,C_2>0$.
\end{prop}

We remark here that the lower bound in \eqref{eq:vgam} is useful only
when $\tau \le d-1$. We believe that a positive lower bound still
exists when $\tau>d-1$, even though our arguments in general do not
apply for such $\tau$.

In the case of a deterministic speed $v$, Proposition~\ref{prop:var}
provides an explicit condition on $\theta$ ensuring that the variance
scales like the volume of the observation window in the classical
Johnson--Mehl growth model.  The problem of finding such a condition,
explicitly formulated in \cite[page~754]{CQ01}, arose in \cite{CQ97},
where asymptotic normality for the number of exposed seeds in a
region, as the volume of the region approaches infinity, is obtained
under the assumption that the variance scales properly.  This was by
then only shown numerically for the case when $\theta$ is the Lebesgue
measure and $d=1,2,3,4$.  Subsequent papers \cite{PY02, SY08} derived
the variance scaling for $\theta$ being the Lebesgue measure and some
generalizations of it, but in a slightly different formulation of the
model, in which seeds that do not appear in the observation window are
automatically rejected and cannot influence the growth pattern in the
region $W$.

It should be noted that it might also be possible to use Theorem~1.2
in \cite{MR4019871} to obtain a quantitative CLT and variance
asymptotics for statistics of the exposed points in a domain $W$ which
is the union of unit cubes around a subset of points in
$\mathbb{Z}^d$. For this, one would need to check Assumption~1.1 from
the cited paper, which ensures non-degeneracy of the variance, and a
moment condition in the form of equation~(1.10) therein. It seems to
us that checking Assumption~1.1 can be a challenging task and would
involve further assumptions on the model, such as the one we also need
in our Proposition~\ref{prop:var}. Controls on the long-range
interactions would also be necessary to check (1.10). Thus, while this
might indeed yield results similar to us, the goal of the present work
is to highlight the application of region-stabilization in this
context, which in general is of a different nature from the methods in
\cite{MR4019871}. For example, the approach in \cite{MR4019871} does
not apply for Pareto minimal points in a hypercube considered in
\cite{bha:mol2021}, since there is no polynomial decay in long-range
interactions, while region-stabilization yields optimal rates for the
Gaussian convergence.

The bounds in Theorem~\ref{thm:birth} can be specified under two
different scenarios. When considering a sequence of weight functions,
under suitable conditions Theorem~\ref{thm:birth} provides a
quantitative CLT for the corresponding functionals $(F_n)_{n \in
  \N}$. Keeping all other quantities fixed with respect to $n$,
consider the sequence of non-negative location-weight functions on
$\R^d$ given by $h_{1,n} = \ind_{n^{1/d} W}$ for a fixed convex body
$W \subset \R^d$ satisfying (A).
In view of Proposition~\ref{prop:var}, this provides
the following quantitative CLT.

\begin{theorem}\label{thm:scale}
  Let the assumptions (A)--(C) be in force.  For $n \in \N$ and $\eta$
  as in Theorem~\ref{thm:birth}, let $F_n:=F_n(\eta)$, where $F_n$ is
  defined as in \eqref{eq:F} with $a$ independent of $n$ and
  $h_1=h_{1,n} = \ind_{n^{1/d}W}$.  Assume that $\theta$ and $\nu$
  satisfy
  \begin{equation}\label{eq:thcond}
    \int_{0}^a w(t) \theta(\diff t)
    - 2 \bvol \spm  \int_0^a \int_0^{t} (t-s)^d w(s)
    w(t)
    \theta(\diff s) \theta(\diff t)>0\,,
  \end{equation}
  where $w(t)$ is given at \eqref{eq:6}. Then there exists a constant
  $C \in (0,\infty)$, depending on $a$, $d$, the first $7d$ moments of
  $\nu$, $\theta$ and $W$, such that
  \begin{displaymath}
    \max\left\{d_{\mathrm{W}}\left(\frac{F_n - \E F_n}{\sqrt{\Var
            F_n}},N\right),d_{\mathrm{K}}\left(\frac{F_n - \E F_n}{\sqrt{\Var
            F_n}},N\right)\right\}
    \le C n^{-1/2} 
  \end{displaymath}
  for all $n\in\N$. In particular, \eqref{eq:thcond} is
  satisfied for $\theta$ given at \eqref{eq:10} with
  $\tau \in (-1,d-1]$. 
  
  Furthermore, the bound on the Kolmogorov distance is of optimal
  order, i.e., when \eqref{eq:thcond} holds, there exists a constant
  $0 < C' \le C$ depending only on $a$, $d$, the first $2d$ moments of
  $\nu$, $\theta$ and $W$, such that
  \begin{displaymath}
    d_{\mathrm{K}}\left(\frac{F_n - \E F_n}{\sqrt{\Var F_n}},N\right)
    \ge C' n^{-1/2}.
  \end{displaymath}
\end{theorem}


When \eqref{eq:thcond} is satisfied, Theorem~\ref{thm:scale} yields a
CLT for the number of exposed seeds born before time
$a\in (0,\infty)$, with rate of convergence of order $n^{-1/2}$.  This
extends the CLT for the number of exposed seeds from \cite{CQ97} in
several directions: the model is generalized to random growth speeds,
there is no constraint of any kind on the shape of the window
$W$ except convexity, and a logarithmic factor is removed
from the rate of convergence.

In a different scenario, if $\theta$ has a power-law density
\eqref{eq:10} with $\tau \in (-1,d-1]$, it is possible to explicitly
specify the dependence of the bound in Theorem~\ref{thm:birth} on the
moments of $\nu$, as stated in the following result. Note that for the
above choice of $\theta$, the assumption (B) is trivially
satisfied. Denote
$$
V_\nu(W):=\sum_{i=0}^d V_{d-i}(W) \spm[d+i],
$$
which is the sum of the intrinsic volumes of $W$ weighted by the
moments of the speed.

\begin{theorem}
  \label{thm:speed}
  Let the assumptions (A) and (C) be in force.  For $\theta$ given at
  \eqref{eq:10} with $\tau \in (-1,d-1]$, consider $F=F(\eta)$, where
  $\eta$ is as in Theorem~\ref{thm:birth} and $F$ is defined as in
  \eqref{eq:F} with $a \in (0,\infty)$.  Then, there exists a constant
  $C \in (0,\infty)$, depending only on $d$ and $\tau$, such that
  \begin{multline*}
    d_{\mathrm{W}}\left(\frac{F - \E F}{\sqrt{\Var F}},N\right)\\
    \leq
    C (1+a^d) \left(1+\spm[7d]\spm^{-7}\right)^{2}
    \Bigg[
    \frac
    {\spm^{-\frac{1}{2}\left(\frac{\tau+1}{d+\tau+1}+1\right)}
      \sqrt{V_\nu(W)}}{l_{a,\tau}(\nu_d) \lambda(W)}
    + \frac{\spm^{-\frac{\tau+1}{d+\tau+1}-1} V_\nu(W)}
    {l_{a,\tau}(\nu_d)^{3/2}\lambda(W)^{3/2}}
    \Bigg]\,,
  \end{multline*}
  and
  \begin{multline*}
    d_{\mathrm{K}}\left(\frac{F - \E F}{\sqrt{\Var F}},N\right)
    \leq  C (1+a^d)^{3/2} \left(1+\spm[7d]\spm^{-7}\right)^{2}
    \Bigg[
    \frac
    {\spm^{-\frac{1}{2}\left(\frac{\tau+1}{d+\tau+1}+1\right)}
      \sqrt{V_\nu(W)}}{l_{a,\tau}(\nu_d) \lambda(W)}
    \\
    + \frac{\spm^{-\frac{\tau+1}{d+\tau+1}-1} V_\nu(W)}
    {l_{a,\tau}(\nu_d)^{3/2}\lambda(W)^{3/2}}  + \frac
    {\spm^{-\frac{5}{4}\left(\frac{\tau+1}{d+\tau+1}+1\right)}V_\nu(W)^{5/4} + \spm^{-\frac{3}{2}\left(\frac{\tau+1}{d+\tau+1}+1\right)} V_\nu(W)^{3/2}}
    {l_{a,\tau}(\nu_d)^2 \lambda(W)^{2}}
    \Bigg],
  \end{multline*}
  where $l_{a,\tau}$ is defined at \eqref{eq:la}.
\end{theorem} 

Note that our results for the number of exposed points can also be interpreted as quantitative central limit theorems for the number of local minima of the growth frontier, which is of independent interest. As an application of Theorem~\ref{thm:speed}, we consider the case
when the intensity of the underlying point process grows to
infinity. The quantitative central limit theorem for this case is
contained in the following result.

%

\begin{corollary}
  \label{cor:scale2}
  Let the assumptions (A) and (C) be in force. Consider $F(\eta_s)$
  defined at \eqref{eq:F} with $a \in (0,\infty)$, evaluated at the
  Poisson process $\eta_s$ with intensity
  $s\lambda\otimes\theta\otimes\nu$ for $s \ge 1$ and $\theta$ given
  at \eqref{eq:10} with $\tau \in (-1,d-1]$.  Then, there exists a
  finite constant $C \in (0,\infty)$ depending only on $W$, $d$, $a$,
  $\tau$, $\spm$, and $\spm[7d]$, such that, for all $s\ge 1$,
  \begin{align*}
    \max\left\{
    d_{\mathrm{W}}\left(\frac{F(\eta_s) - \E F(\eta_s)}
    {\sqrt{\Var F(\eta_s)}},N\right),
    d_{\mathrm{K}}\left(\frac{F(\eta_s) - \E F(\eta_s)}
    {\sqrt{\Var F(\eta_s)}},N\right)
    \right\}\leq
    C s^{-\frac{d}{2(d+\tau+1)}}\,.
  \end{align*}
  Furthermore, the bound on the Kolmogorov distance is of optimal
  order.
\end{corollary}

\section{Variance estimation}
\label{sec:variance-estimation}

In this section, we estimate the mean and variance of the statistic
$F$, thus providing a proof of Proposition~\ref{prop:var}. Recall the
weight function $h(\bx):=h_1(x) h_2(t_x)$, where
$h_1(x)=\ind\{x \in W\}$ and $h_2(t) = \ind\{t \le
a\}$. Notice that by the Mecke formula, 
the mean of $F$ is given by
\begin{align*}\label{eq:exp}
  \E F(\eta)&=\int_\XX h(\bx)\E H_\bx(\eta+\delta_\bx)\mu(\diff \bx) \nonumber\\
  &=\int_{\R^d}h_1(x) \diff x \int_0^\infty h_2(t) w(t) 
    \theta(\diff t)
    =\lambda(W)\int_0^a w(t) 
    \theta(\diff t),
\end{align*}
where $w(t)$ is defined at \eqref{eq:6}. In many instances,
we will use the simple inequality
\begin{equation}
  \label{eq:8}
  2ab\leq a^2 +b^2, \quad a,b \in \R_+.
\end{equation}
Also notice that for $x \in \R^d$,
\begin{equation}\label{eq:intball}
  \int_{\R^d} \lambda\big(B_{r_1}(0)\cap B_{r_2}(x)\big)\diff x
  =\int_{\R^d} \ind_{y\in B_{r_1}(0)} \int_{\R^d}
  \ind_{y\in B_{r_2}(x)}\diff x \diff y=\bvol^2r_1^dr_2^d\,.
\end{equation}
The multivariate Mecke formula (see, e.g., \cite[Th.~4.4]{last:pen})
yields that
\begin{multline*}
  \Var(F)=\int_\XX h(\bx)^2 \E H_{\bx}(\eta+\delta_\bx) \mu(\diff \bx)
  -\Big(\int_\XX h(\bx) \E H_{\bx}(\eta+\delta_\bx) \mu(\diff \bx)\Big)^2
  \\
  +\int_{D} h(\bx)h(\by)
  \E \big[H_{\bx}(\eta+\delta_{\by}+\delta_\bx)
  H_{\by}(\eta+\delta_{\bx}+\delta_\by)\big]
  \mu^2(\diff( \bx, \by)),
\end{multline*}
where the double integration is over the region $D\subset \XX$ where
the points $\bx$ and $\by$ are incomparable ($\bx \not \preceq \by$
and $\by \not \preceq \bx$), i.e.,
$$
D:=\big\{(\bx,\by) : \|x-y\| > \max\{v_x(t_y-t_x), v_y(t_x-t_y)\}\big\}.
$$
It is possible to get rid of one of the Dirac measures in the inner
integral, since on $D$ the points are incomparable. Thus, using the
translation invariance of $\E H_{\bx}(\eta)$, we have
\begin{equation}\label{eq:Var}
  \Var(F)= \lambda(W) \int_{0}^a w(t) 
  \theta(\diff t) - I_0+ I_1,
\end{equation}
where 
$$
I_0:=2 \int_{\XX^2} \ind_{\by \preceq \bx} h_1(x) h_1(y) h_2(t_x) h_2(t_y) w(t_x) w(t_y) 
\mu^2(\diff (\bx,\by)),
$$
and
$$
I_1:=\int_D h_1(x)h_1(y) h_2(t_x) h_2(t_y)\Big[
\E \big[H_{\bx}(\eta+\delta_\bx)H_{\by}(\eta+\delta_\by)\big] - w(t_x) w(t_y)\Big]
\mu^2(\diff (\bx, \by)).
$$
Finally, we will use the following simple inequality for the
incomplete gamma function
\begin{equation}\label{eq:lasc}
  \min\{1,b^x\} \gamma(x,y) \le \gamma(x,by) \le \max\{1,b^x\} \gamma(x,y),
\end{equation}
which holds for all $x \in \R_+$ and $b,y>0$.

\begin{proof}[Proof of Proposition~\ref{prop:var}]
  First, notice that the term $I_1$ in \eqref{eq:Var} is non-negative,
  since
  \begin{displaymath}
    \E [H_{\bx}(\eta)H_{\by}(\eta)]
    = e^{-\mu(L_{\bx} \cup L_{\by})}
    \ge e^{-\mu(L_{\bx})}e^{-\mu(L_{\by})}=w(t_x) w(t_y).
  \end{displaymath}
  Furthermore, \eqref{eq:8} yields that
  \begin{align*}
    I_0 \le \int_\XX & h_1(x)^2 h_2(t_x) w(t_x)
    \left[\int_{\XX} \ind_{\by \preceq \bx} h_2(t_y) w(t_y) \mu(\diff
      \by)\right]
    \mu(\diff \bx)\\
    & + \int_\XX h_1(y)^2 h_2(t_y) w(t_y) 
    \left[\int_{\XX} \ind_{\by \preceq \bx} h_2(t_x) w(t_x)
      \mu(\diff \bx)\right]\mu(\diff \by).
  \end{align*}
  Since $\by \preceq \bx$ is equivalent to $\|y-x\| \le v_y(t_x-t_y)$,
  the first summand on the right-hand side above can be simplified as
  \begin{align*}
    &\int_\XX h_1(x)^2 h_2(t_x) w(t_x) \left[\int_{\XX} \ind_{\by \preceq \bx} h_2(t_y) w(t_y) \mu(\diff \by)\right] \mu(\diff \bx) \\
	&=\int_{\R^d}\int_{0}^\infty h_1(x)^2 h_2(t_x) w(t_x)  \theta(\diff t_x) \diff x \int_0^\infty \int_{0}^{t_x} \bvol v_y^d (t_x-t_y)^d h_2(t_y) w (t_y) \theta(\diff t_y) \nu(\diff v_y)\\
	& =\bvol \spm  \lambda(W) \int_0^a \int_0^{t} (t-s)^dw(s) w(t) \theta(\diff s) \theta(\diff t).
  \end{align*}
  The second summand in the bound on $I_0$, upon interchanging
  integrals for the second step, turns into 
  \begin{align*}
    &\int_\XX h_1(y)^2 h_2(t_y) w(t_y) 
    \left[\int_{\XX} \ind_{\by \preceq \bx} h_2(t_x) w(t_x)
    \mu(\diff \bx)\right]\mu(\diff \by)\\
	&=\int_{\R^d}\int_{0}^\infty h_1(y)^2 h_2(t_y) w(t_y) \theta(\diff t_y) \diff y \int_0^\infty \int_{t_y}^{\infty} \bvol v_y^d (t_x-t_y)^d h_2(t_x) w (t_x) \theta(\diff t_x) \nu(\diff v_y)\\
	& =\bvol \spm   \lambda(W)  \int_0^a \int_0^{t} (t-s)^d  w(s) w(t) \theta(\diff s) \theta(\diff t).
  \end{align*}
  Combining, by \eqref{eq:Var} we obtain \eqref{eq:9}.
  
  To prove \eqref{eq:vub}, note that by the Poincar\'e inequality (see
  \cite[Sec.~18.3]{last:pen}),
  $$
  \Var(F) \le \int_\XX \E\big(F(\eta + \delta_{\bx})- F(\eta)\big)^2 \mu(\diff \bx).
  $$
  Observe that $\eta$ is simple and for $x\notin\eta$
  \begin{align*}
  F(\eta + \delta_\bx)- F(\eta)=h(\bx) H_\bx(\eta+\delta_\bx)
  - \sum_{\by \in \eta} h(\by) H_\by(\eta)\ind_{\by \succ \bx}\,. 
  \end{align*}
  The inequality
  $$
  -\sum_{\by \in \eta} h(\bx) h(\by) H_\bx(\eta+\delta_\bx) H_\by(\eta)\ind_{\by \succ \bx}\leq 0
  $$
  in the first step and the Mecke formula in the second step yield that
  \begin{align}\label{eq:po}
  &\int_\XX \E\big|F(\eta + \delta_{\bx})- F(\eta)\big|^2 \mu(\diff \bx)\nonumber\\
	&\leq
	\int_{\XX} \E\big[h(\bx)^2 H_\bx(\eta+\delta_\bx)\big]\mu(\diff\bx)
	+\int_{\XX} \E\Big[\sum_{\by,\bz \in \eta}
	\ind_{\by \succ \bx}\ind_{\bz \succ \bx}
	h(\by)h(\bz) H_\by(\eta)H_\bz(\eta)\Big]\mu(\diff\bx)\nonumber
	\\
	&=\int_\XX h(\bx)^2 w(t_x) \mu(\diff \bx) + \int_{\XX^2} \ind_{\by \succ \bx} h(\by)^2 w(t_y) \mu^2(\diff (\bx,\by))\nonumber \\
	&\qquad\qquad\qquad\qquad \qquad\qquad+ \int_\XX \int_{D_\bx} h(\by) h(\bz) e^{-\mu(L_{\by} \cup L_{\bz})} \mu^2(\diff (\by,\bz)) \mu(\diff \bx),
  \end{align}
  where 
  $$
  D_\bx:=\big\{(\by ,\bz) \in \XX^2 : \by \succ \bx, \bz \succ \bx, \by  \not \succ \bz , \bz  \not \succ \by\big\}.
  $$
  Using that $x e^{-x/2} \le 1$ for $x \in \R_+$, observe that
  \begin{align}
    \int_{\XX^2} \ind_{\by \succ \bx} h(\by)^2 w(t_y) \mu^2(\diff
    (\bx,\by))&= \int_\XX h(\by)^2 w(t_y) \mu(L_\by) \mu(\diff \by)
                \nonumber\\
    &\le
    \int_\XX h(\by)^2 w(t_y)^{1/2} \mu(\diff \by). \label{eq:po1}
  \end{align}
  Next, using that $\mu(L_{\by} \cup L_{\bz}) \ge (\mu(L_{\by}) +
  \mu(L_{\bz}))/2$ and that $D_\bx \subseteq \{\by,\bz \succ \bx\}$
  for the first inequality, and \eqref{eq:8} for the second one, we have
  \begin{align}\label{eq:vu}
  &\int_\XX \int_{D_\bx} h(\by) h(\bz) e^{-\mu(L_{\by} \cup L_{\bz})} \mu^2(\diff (\by,\bz)) \mu(\diff \bx) \nonumber\\
  &\le \int_\XX \int_{\XX^2} \ind_{\by,\bz \succ \bx} h(\by) h(\bz) w(t_y)^{1/2}w(t_z)^{1/2} \mu^2(\diff (\by,\bz)) \mu(\diff \bx)\nonumber\\
  &\le \int_{[0,a]^2}  w(t_y)^{1/2}w(t_z)^{1/2} \int_{\R^{2d}} h_1(z)^2 \left(\int_{\XX} \ind_{\bx \prec \by,\bz} \mu(\diff \bx)\right) \diff(y,z) \theta^2(\diff(t_y,t_z)).
  \end{align}
  By \eqref{eq:intball},
  \begin{align*}
  \int_{\R^d} \int_{\XX} \ind_{\bx \prec \by,\bz} \mu(\diff \bx) \diff y &=\int_{0}^{t_y \wedge t_z} \int_{0}^\infty \nu(\diff v_x) \theta(\diff t_x)  \int_{\R^d} \lambda\big(B_{v_x(t_y-t_x)}(y) \cap B_{v_x(t_z-t_x)}(z)\big) \diff y\\
  &=\bvol^2 \nu_{2d} \int_{0}^{t_y \wedge t_z}(t_y-t_x)^d (t_z-t_x)^d \theta(\diff t_x).
  \end{align*}
  Plugging in \eqref{eq:vu}, we obtain
  \begin{multline*}
  \int_\XX \int_{D_\bx} h(\by) h(\bz) e^{-\mu(L_{\by} \cup L_{\bz})} \mu^2(\diff (\by,\bz)) \mu(\diff \bx) \\
  \le  \bvol^2 \nu_{2d} \lambda(W) \int_{[0,a]^2} \int_{0}^{t_1 \wedge t_2}(t_1-s)^d (t_2-s)^d w(t_1)^{1/2}w(t_2)^{1/2}   \theta(\diff s) \theta^2(\diff(t_1,t_2)).
  \end{multline*}
  This in combination with \eqref{eq:po} and \eqref{eq:po1} proves \eqref{eq:vub}. 

  Now we move on to prove \eqref{eq:vgam}. We first confirm the lower bound. Fix $\tau \in (-1,d-1]$, as otherwise the bound is trivial, and $a \in (0,\infty)$. Then 
  $$
  \Lambda(t)=\bvol \int_0^t (t-s)^d s^\tau \diff s
  = \bvol t^{d+\tau+1} B(d+1,\tau+1)=B\, \bvol t^{d+\tau+1},
  $$
  where $B:=B(d+1,\tau+1)$ is a value of the Beta function.  Hence,
  $w(t)=\exp\{-B \,\bvol \spm t^{d+\tau + 1}\}$.  Plugging in, we obtain
  \begin{align}\label{eq:Varb}
    &\frac{\Var(F)}{\lambda(W)} \ge \int_{0}^a e^{-B \,\bvol \spm  t^{d+\tau + 1}} \theta(\diff t) - 2 \bvol \spm  \int_0^a \int_0^{t} (t-s)^d e^{-B \,\bvol \spm  (s^{d+\tau +1}+t^{d+\tau+1})} \theta(\diff s) \theta(\diff t)\nonumber\\
	&\quad\, =\left(\frac{1}{B\,\bvol \spm }\right)^{\frac{\tau+1}{d+\tau+1}}  \Bigg[\int_{0}^b  e^{- t^{d+\tau+1}}t^{\tau} \diff t 
	-\frac{2}{B} \int_0^b \int_0^{t} (t-s)^d e^{-(s^{d+\tau+1}+t^{d+\tau+1})} t^{\tau} s^{\tau} \diff s \diff t\Bigg],
  \end{align}
  where $b:=a(B \,\bvol \spm)^{1/(d+\tau+1)}$.
  Writing $s=tu$ for some $u \in [0,1]$, 
  \begin{align*}
    &\frac{2}{B} \int_0^b \int_0^{t} (t-s)^d e^{-(s^{d+\tau+1}+t^{d+\tau+1})} t^{\tau} s^{\tau}\diff s \diff t\\
    & \le \frac{2}{B} \int_0^b t^{d+2\tau+1} \int_0^{1} (1-u)^d u^{\tau} e^{-t^{d+\tau+1}(u^{d+\tau+1}+1)} \diff u \diff t < 2\int_0^b t^{d+2\tau+1} e^{-t^{d+\tau+1}} \diff t.
  \end{align*}
  By substituting $t^{d+\tau+1}=z$, it is easy to check that for any
  $\rho>-1$,
  $$
  \int_{0}^b e^{- t^{d+\tau+1}}t^\rho \diff t=\frac{1}{d+\tau+1} \gamma\left(\frac{\rho+1}{d+\tau+1},b^{d+\tau+1}\right),
  $$
  where $\gamma$ is the lower incomplete Gamma function.
  In particular, 
  using that $x\gamma(x,y)>\gamma(x+1,y)$ for $x,y>0$ we have
  \begin{align*}
  \int_{0}^b e^{- t^{d+\tau+1}}t^{d+2\tau+1} \diff t&=\frac{1}{d+\tau+1} \gamma\left(1+\frac{\tau+1}{d+\tau+1},b^{d+\tau+1}\right)\\
  &<\frac{\tau+1}{(d+\tau+1)^2} \gamma\left(\frac{\tau+1}{d+\tau+1},b^{d+\tau+1}\right).
  \end{align*}
  Thus, since $\tau\in (-1,d-1]$,
  \begin{align*}
  &\int_{0}^b e^{- t^{d+\tau+1}}t^\tau \diff t -\frac{2}{B} \int_0^b \int_0^{t} (t-s)^d e^{-(s^{d+\tau+1}+t^{d+\tau+1})} t^\tau s^\tau \diff s \diff t\\
  &>\gamma\left(\frac{\tau+1}{d+\tau+1},b^{d+\tau+1}\right) \frac{1}{d+\tau+1}\left[1- \frac{2(\tau+1)}{d+\tau+1} \right]\ge 0.
  \end{align*}
  By \eqref{eq:Varb} and \eqref{eq:lasc}, we obtain
the lower bound in \eqref{eq:vgam}.

For the upper bound in \eqref{eq:vgam}, for $\theta$ as in
\eqref{eq:10}, arguing as above we have
$$
\int_{0}^a w(t)^{1/2} \theta(\diff t)=\int_{0}^a e^{-B \,\bvol \spm  t^{d+\tau + 1}/2} \theta(\diff t)= \frac{\left(2/B\,\bvol \spm \right)^{\frac{\tau+1}{d+\tau+1}}}{d+\tau+1} \gamma\left(\frac{\tau+1}{d+\tau+1},b^{d+\tau+1}\right).
$$
Finally, substituting $s'=(B \,\bvol \spm)^{\frac{1}{d+\tau+1}} s$ and
similarly for $t_1$ and $t_2$, it is straightforward to see that
\begin{align*}
  \nu_{2d}  \int_{[0,a)^2} &\int_{0}^{t_1 \wedge t_1} (t_1-s)^d (t_2-s)^d w(t_1)^{1/2}w(t_2)^{1/2} \theta(\diff s) \theta^2(\diff(t_1,t_2)) \\
                           &\le C \nu_{2d} \nu_d^{-2} \nu_d^{-\frac{\tau+1}{d+\tau+1}} \left(\int_{\R_+} t^{d+\tau} e^{-t^{d+\tau+1}/4} \diff t\right)^2 \int_{0}^{b} s'^\tau e^{-s'^{d+\tau+1}/2} \diff s'\\
                           & \le C' \nu_{2d} \nu_d^{-2} \nu_d^{-\frac{\tau+1}{d+\tau+1}}   \gamma\left(\frac{\tau+1}{d+\tau+1},\frac{b^{d+\tau+1}}{2}\right)
\end{align*}
for some constants $C,C'$ depending only on $d$ and $\tau$. The upper
bound in \eqref{eq:vgam} now follows from \eqref{eq:vub} upon using
the above computation and \eqref{eq:lasc}.
\end{proof}

\section{Proofs of the Theorems}
\label{sec:proof-theor-refthm:b}

In this section, we derive our main results using Theorem~2.1 in
\cite{bha:mol2021}. While we do not restate this theorem here and
refer the reader to Section~2 in \cite{bha:mol2021}, it is important
to note that the Poisson process considered in
\cite[Theorem~2.1]{bha:mol2021} has the intensity measure $s\QQ$
obtained by scaling a fixed measure $\QQ$ on $\XX$ with
$s$. Nonetheless, the main result is non-asymptotic and, while in the
current paper, we consider a Poisson process with fixed intensity
measure $\mu$ (without a scaling parameter), we can still use
\cite[Theorem~2.1]{bha:mol2021} with $s=1$ and the measure $\QQ$
replaced by $\mu$. While still following the notation from
\cite{bha:mol2021}, we drop the subscript $s$ for ease of
notation.

Recall that for $\sM \in \Nb$, the score function $\xi(\bx,\sM)$ is
defined at \eqref{eq:4}.  It is straightforward to check that if
$\xi (\bx, \sM_1)=\xi (\bx, \sM_2)$ for some $\sM_1,\sM_2 \in \Nb$ with
$0\neq \sM_1 \leq \sM_2$ (meaning that $\sM_2-\sM_1$ is a nonnegative
measure) and $\bx \in \sM_1$, then $\xi (\bx, \sM_1)=\xi (\bx, \sM)$ for
all $\sM\in\Nb$ such that $\sM_1\leq \sM\leq \sM_2$, so that
equation (2.1) in \cite{bha:mol2021} holds.
Next we check assumptions
(A1) and (A2) in \cite{bha:mol2021}.


For $\sM \in \Nb$ and $x\in\sM$, define the stabilization region
\begin{displaymath}
  R(\bx,\sM):=
  \begin{cases}
    L_{x,t_x} & \text{if}\; \bx \;\text{is exposed in
      $\sM$},\\
    \emptyset & \text{otherwise}.
  \end{cases}
\end{displaymath}
Notice that
\begin{displaymath}
	\{\sM\in\Nb\colon \by\in R(\bx,\sM+\delta_\bx)\}
	\in\mathscr{N}\quad\text{for all}\ \bx,\by\in\XX,
\end{displaymath}  
and that
\begin{displaymath}
	\Prob{\by\in R(\bx,\eta + \delta_{\bx})}
	=\ind_{\by\preceq \bx} e^{-\mu(L_{x,t_x})}
	=\ind_{\by\preceq \bx} w(t_x),
\end{displaymath}
and
\begin{displaymath}
	\P\{\{\by,\bz\} \subseteq R(\bx, \eta +\delta_{\bx})\}
	=\ind_{\by\preceq \bx}\ind_{\bz\preceq\bx} e^{-\mu(L_{x,t_x})}
	=\ind_{\by\preceq \bx}\ind_{\bz\preceq\bx} w(t_x)
\end{displaymath} 
are measurable functions of $(\bx,\by)\in\XX^2$
and $(\bx,\by,\bz)\in\XX^3$ respectively,
with $w(t)$ defined at \eqref{eq:6}.
It is not hard to see that $R$ is monotonically decreasing in the
second argument, and that for all $\sM\in\Nb$ and $\bx\in\sM$,
$\sM(R(x,\sM))\geq 1$ implies that $\bx$ is exposed, hence
$(\sM+\delta_\by)(R(x,\sM+\delta_\by))\geq1$
for all $\by\not\in R(\bx,\sM)$.
Moreover, the function $R$ satisfies
\begin{displaymath}
  \xi\big(\bx,\sM\big)=\xi\Big(\bx,\sM_{R(\bx,\sM)}\Big),
  \quad \sM\in\Nb,\; \bx\in\sM\,,
\end{displaymath}
where $\sM_{R(\bx,\sM)}$ denotes the restriction of the measure $\sM$
to the region $R(\bx,\sM)$. It is important to note here that this
holds even when $\bx$ is not exposed in $\sM$, since in this case, the
left-hand side is $0$ where the right-hand side is $0$ by our
convention that $\xi(\bx,0)=0$.  Hence, assumptions (A1.1)--(A1.4) in
\cite{bha:mol2021} are satisfied.  Further, notice that for any
$p \in (0,1]$, for all $\sM\in\Nb$ with $\sM(\XX) \le 7$, we have
\begin{displaymath}
  \E\left[\xi(\bx,\eta+\delta_{\bx}+\sM)^{4+p}\right]
  \leq \ind_{x \in W}\ind_{t_x \le a} w(t_x),
\end{displaymath}
confirming condition (A2) in \cite{bha:mol2021} with
$M_p(\bx):=\ind\{x \in W, t_x \le a\}$. For definiteness, we take
$p=1$, and define
\begin{displaymath}
  \widetilde{M}(\bx):=\max\{M_1(\bx)^2,M_1(\bx)^4\}
  =\ind_{x \in W}\ind_{t_x\leq a}.
\end{displaymath}
Finally, define
\begin{equation*}
  r(\bx,\by):=
  \begin{cases}
    \spm \Lambda(t_x),\ &\text{if}\ \by\preceq \bx,\\
    \infty,\ &\text{if}\ \by\not\preceq \bx,
  \end{cases}
\end{equation*}
so that 
\begin{displaymath}
  \Prob{\by\in R(\bx,\eta + \delta_{\bx})}
  =\ind_{\by\preceq \bx} w(t_x)
  =e^{-r(\bx,\by)},\quad \bx,\by\in\XX,
\end{displaymath}
which corresponds to equation (2.4) in \cite{bha:mol2021}.  
Now that we have checked all the necessary conditions, we can invoke
Theorem~2.1 in \cite{bha:mol2021}. Let $\zeta:=\frac{p}{40+10p}=1/50$,
and define functions of $\by\in\XX$ by 
\begin{align}
  \label{eq:g}
  g(\by) :=
  & \int_{\XX} e^{-\zeta r(\bx, \by)} \,\mu(\diff \bx),\\
  \label{eq:g1}
  h(\by):=
  &\int_{\XX} \ind_{x \in W}\ind_{t_x\le a}
  e^{-\zeta r(\bx, \by)} \,\mu(\diff \bx),\\ 
  \label{eq:G}
  G(\by) :=
  & \ind_{y \in W}\ind_{t_x\leq a}
    + 
    \max\{h(\by)^{4/9}, h(\by)^{8/9}\} 
    \big(1+g(\by)^4\big).
\end{align}
For $\bx,\by \in \XX$, denote
\begin{equation}
  \label{eq:q}
  q(\bx,\by):=\int_\XX \P\Big\{\{\bx,\by\} 
  \subseteq R\big(\bz, \eta +\delta_{\bz}\big)\Big\} \,\mu(\diff \bz)
  =\int_{\bx\preceq \bz,\by\preceq \bz} w(t_z) \,\mu(\diff \bz).
\end{equation}
For $\alpha>0$, let
\begin{displaymath}
  f_\alpha(\by):=f_\alpha^{(1)}(\by)+f_\alpha^{(2)}(\by)+f_\alpha^{(3)}(\by),
\quad \by\in\XX,
\end{displaymath}
where for $\by \in \XX$,
\begin{align}
  \label{eq:fa1}
  f_\alpha^{(1)}(\by)&:=\int_\XX G(\bx) e^{- \alpha r(\bx,\by)}
                       \;\mu(\diff \bx)=\int_{\by\preceq\bx} G(\bx)
                       w(t_x)^\alpha \mu(\diff \bx), \nonumber\\
  f_\alpha^{(2)} (\by)&:=\int_\XX G(\bx) e^{- \alpha r(\by,\bx)} 
                        \;\mu(\diff \bx)=w(t_y)^\alpha
                        \int_{\bx\preceq\by}G(\bx)\mu(\diff \bx), \nonumber\\
  f_\alpha^{(3)}(\by)&:=\int_{\XX} G(\bx) q(\bx,\by)^\alpha \;\mu(\diff \bx).
\end{align}
Finally, let
\begin{displaymath}
  \kappa(\bx):= \Prob{\xi(\bx, \eta+\delta_{\bx}) \neq 0}
  = \ind_{x\in W}\ind_{t_x\le a}
  w(t_x),\quad
  \bx\in\XX. 
\end{displaymath}
For an integrable function $f : \XX \to \R$, denote
$\mu f:=\int_\XX f(\bx) \mu(\diff \bx)$. With
$\beta:=\frac{p}{32+4p}=1/36$, \cite[Theorem~2.1]{bha:mol2021} yields
that $F=F(\eta)$ as in \eqref{eq:F} satisfies
\begin{equation}\label{eq:Wass}
  d_{\mathrm{W}}\left(\frac{F-\E F}{\sqrt{\Var F}},  N\right) 
  \leq C \Bigg[\frac{\sqrt{ \mu f_\beta^2}}{\Var F}
  +\frac{ \mu ((\kappa+g)^{2\beta}G)}{(\Var F)^{3/2}}\Bigg],
\end{equation}
and
\begin{multline}\label{eq:Kol}
  d_{\mathrm{K}}\left(\frac{F-\E F}{\sqrt{\Var F}},
    N\right) 
  \leq C \Bigg[\frac{\sqrt{\mu f_\beta^2}
    + \sqrt{\mu f_{2\beta}}}{\Var F}
  +\frac{\sqrt{ \mu ((\kappa+g)^{2\beta}G)}}{\Var F}\\
  +\frac{ \mu ((\kappa+g)^{2\beta} G)}{(\Var F)^{3/2}}
  +\frac{( \mu ((\kappa+g)^{2\beta} G))^{5/4}
    + ( \mu ((\kappa+g)^{2\beta} G))^{3/2}}{(\Var F)^{2}}\Bigg],
\end{multline}
where $N$ is a standard normal random variable and
$C \in (0,\infty)$ is a constant.

In the rest of this section, we estimate the summands on the
right-hand side of the above two bounds to obtain our main
results. While the bounds above are admittedly quite difficult to
interpret, they essentially involve integrals of functions which are
products involving an exponential part and a polynomial part. Because
of a faster decay of the exponential part, the integrals grow at a
rate that is at most some small enough power of the variance of $F$,
and this yields the presumably optimal rates of convergence in
Theorem~\ref{thm:scale}.
We start with a simple lemma.


\begin{lemma}
  \label{lemma:psi}
  For all $x \in \R_+$ and $y>0$,
  \begin{equation}\label{eq:1a}
    Q(x,y):=\int_0^\infty t^x \, e^{-y\Lambda(t)}\, \theta(\diff t)
    = \int_0^\infty t^x \, w(t)^{y/\nu_d}\, \theta(\diff t) <\infty.
  \end{equation}
\end{lemma}
\begin{proof}
	Assume that $\theta([0,c])>0$ for some $c \in (0,\infty)$, since otherwise the
	result holds trivially. Notice that
	$$
	\int_0^{2c} t^x  \,e^{-y\Lambda(t)}\theta(\diff t) \le (2c)^x
	\int_0^{\infty} e^{-y\Lambda(t)}\theta(\diff t)<\infty
	$$
	by assumption (B). Hence, it suffices to show the finiteness of
	the integral over $[2c,\infty)$. The inequality
	$w^{x/d}e^{-w/2}\leq C$ for some finite constant $C>0$ yields that
	\begin{displaymath}
	\int_{2c}^\infty t^x \, e^{-y \Lambda(t)}\theta(\diff t)
	\leq \frac{C}{y^{x/d}}\int_{2c}^\infty \frac{t^x}{\Lambda(t)^{x/d}} \,
	e^{-y\Lambda(t)/2}\theta(\diff t).
	\end{displaymath}
	For $t \ge 2c$, 
	\begin{displaymath}
	\Lambda(t)=\int_0^t(t-s)^d\theta(\diff s)
	\geq \int_0^{t/2}(t-s)^d\theta(\diff s)
	\geq (t/2)^d\theta([0,t/2]) \ge 2^{-d}t^d \theta([0,c]).
	\end{displaymath}
	Thus,
	\begin{displaymath}
	\int_{2c}^\infty t^x \, e^{-y \Lambda(t)}\theta(\diff t)
	\leq \frac{C 2^{x}}{(y\theta([0,c]))^{x/d}} \int_{2c}^\infty
	e^{-y\Lambda(t)/2}\theta(\diff t)<\infty
	\end{displaymath}
	by assumption (B), yielding the result.
\end{proof}

To compute the bounds in \eqref{eq:Wass} and \eqref{eq:Kol}, we need
to bound $\mu f_{2\beta}$ and $\mu f_\beta^2$, with
$\beta=1/36$. Nonetheless, we provide bounds on $\mu f_\alpha$ and
$\mu f_\alpha^2$ for any $\alpha>0$.  By Jensen's inequality, it
suffices to bound $\mu f_{\alpha}^{(i)}$ and $\mu (f_\alpha^{(i)})^2$
for $i=1,2,3$. This is the objective of the following three lemmas.


For $g$ defined at \eqref{eq:g}
\begin{align*}
  g(\by)&=\int_\XX \ind_{\by\preceq \bx} w(t_x)^\zeta\mu(\diff \bx)
          =\int_{t_y}^\infty \int_{\R^d} \ind_{x \in B_{v_y(t_x-t_y)}(y)} w(t_x)^\zeta
          \;\diff x \theta(\diff t_x)\\
        &=\bvol v_y^d\int_{t_y}^\infty(t_x-t_y)^d
          w(t_x)^\zeta\, \theta(\diff t_x)\leq \bvol v_y^d\int_0^\infty t_x^d
          w(t_x)^\zeta \,\theta(\diff t_x)
          = \bvol  Q(d,\zeta\spm) \, v_y^d.
\end{align*}
where $Q$ is defined at \eqref{eq:1a}. Similarly, for $h$ as in \eqref{eq:g1} with $a \in (0,\infty)$, we have 
\begin{align*}
  h(\by)&=\int_\XX \ind_{\by\preceq \bx} w(t_x)^\zeta
          \ind_{x\in W}\ind_{t_x\le a} \mu(\diff \bx)\\
	&=\ind_{t_y \le a} \int_{t_y}^a \left(\int_{\R^d} \ind_{x \in
   B_{v_y(t_x-t_y)}(y)}
   \ind_{x\in W} \;\diff x \right) w(t_x)^\zeta
	\; \theta(\diff t_x)\\
	& \le \ind_{y \in W+B_{v_y(a-t_y)}(0)} \int_{t_y}^\infty
   \int_{\R^d}
   \ind_{x \in B_{v_y(t_x-t_y)}(y)} w(t_x)^\zeta
   \;\diff x \theta(\diff t_x) \\
  & \le \ind_{y \in W+B_{v_y(a-t_y)}(0)} \int_{t_y}^\infty
   \bvol v_y^d t_x^d w(t_x)^\zeta	\;\theta(\diff t_x) \\
	&\le  \ind_{y \in W+B_{v_y a}(0)}  \bvol  Q(d,\zeta\spm) \, v_y^d.
\end{align*}
Therefore, the function $G$
defined at \eqref{eq:G} for $a \in (0,\infty)$ is bounded by
\begin{align}\label{eq:Gbd}
  G(\by) &\le \ind_{y \in W} + \ind_{y \in W+B_{v_y a}(0)} (1+\bvol  Q(d,\zeta\spm) \, v_y^d)(1+\bvol^4  Q(d,\zeta\spm)^4 \, v_y^{4d})\nonumber\\
  & \le 6 \bvol^5  \ind_{y \in W+B_{v_y a}(0)}p(v_y)\,,
\end{align}
with 
$$
p(v_y):= 1+Q(d,\zeta\spm)^5 v_y^{5d}.
$$

Define
\begin{displaymath}
  M_u:=\int_0^\infty v^u p(v) \nu(\diff v),\quad u \in \R_+.
\end{displaymath}
In particular,
$$
M_0:=\int_0^\infty p(v) \nu(\diff v)=1+Q(d,\zeta\spm)^5 \spm[5d],
$$
and
$$
M:=M_0 + M_d = \int_0^\infty (1+v^d) p(v) \nu(\diff v) 
= 1+\spm + Q(d,\zeta\spm)^5 (\spm[5d]+\spm[6d]).
$$

Recall $V(W)$ defined at \eqref{eq:V} and denote
$\omega=\max_{0 \le j \le d} \omega_j$. The Steiner formula
(see \cite[Section~4.1]{schn2}) yields that
\begin{align}
  \int_{\R_+}\lambda(W+B_{v_xa}(0)) p(v_x) \nu(\diff v_x)
  &= \sum_{i=0}^d \int_{\R_+}\omega_i v_x^i a^i V_{d-i}(W) \,p(v_x)\nu(\diff v_x) \nonumber\\
  &\le \omega (1+a^d)  \sum_{i=0}^d V_{d-i}(W) M_i \label{eq:Stimp}\\
  & \le c_d (1+a^d) M \,  V(W),\label{eq:Steiner}
\end{align}
with $c_d=(d+1)\omega$, where in the final step we have used the
simple inequality $v_x^a \le 1+v_x^b$ for any $0\le a\le b<\infty$. We
will use this fact many times in the sequel without any
mentioning.

We will also often use the fact that for an increasing function $f$ of
the speed $v$, since $p$ is also increasing, by positive association,
we have
$$
\int_{\R_+} f(v) p(v) \nu(\diff v) \ge M_0 \int_{\R_+} f(v) \nu(\diff v).
$$

\begin{lemma}\label{lem:1}
  For $a \in (0,\infty)$, $\alpha>0$ and $f_\alpha^{(1)}$ defined at
  \eqref{eq:fa1},
  $$
  \int_{\XX} f_\alpha^{(1)}(\by)\mu(\diff \by) \le 
  C_1 \,V(W) \; \text{ and }\;
  \int_{\XX} f_\alpha^{(1)}(\by)^2\mu(\diff \by) \le 
  C_2\,V(W),
  $$
  where
  \begin{align*}
    C_1&:=C  (1+a^d) M \frac{Q(0,\alpha \spm/2)}{\alpha},\\
    C_2&:=C  (1+a^d) \, M_0 M \spm[2d] Q(d,\alpha\spm/2)^2 Q(0,\alpha\spm),
  \end{align*}
  for a constant $C \in (0,\infty)$ depending only on $d$.
  \end{lemma}
\begin{proof}
  Using \eqref{eq:Gbd}, we can write
  \begin{multline*}
    \int_{\XX} f_\alpha^{(1)}(\by)\mu(\diff \by) = \int_{\XX}\int_{\by\preceq \bx} G(\bx)w(t_{x})^\alpha \mu(\diff \bx)\mu(\diff \by)\\
     \le 6 \bvol^5  \spm  \int_{\XX} \Lambda(t_x)  \ind_{x \in W+B_{v_xa}(0)}p(v_x) w(t_{x})^\alpha \mu(\diff \bx) =: 6 \bvol^5  \spm \, I_1,
  \end{multline*}
whence using \eqref{eq:Steiner} and that $xe^{-x/2} \le 1$ for $x \in \R_+$, we obtain
\begin{align*}
	I_1 :&=\int_{\R_+^2}\lambda(W+B_{v_xa}(0)) p(v_x) \Lambda(t_x)w(t_{x})^\alpha \theta(\diff t_x) \nu(\diff v_x)\\
	&\le c_d (1+a^d) M \,  V(W) \int_{\R_+}  \Lambda(t_x)w(t_{x})^\alpha \theta(\diff t_x) \\
	&\le c_d (1+a^d) M  \frac{Q(0,\alpha \spm/2)}{\alpha \spm}\,  V(W),
\end{align*}
proving the first assertion.

For the second assertion, first by \eqref{eq:intball}, for any $t_1,t_2 \in \R_+$ we have
\begin{align}\label{eq:muint}
  \int_{\R^d}&\mu(L_{0,t_1}\cap L_{x,t_2})\diff x
               =\int_0^{t_1\wedge t_2} \theta(\diff s)\int_0^\infty \nu(\diff v)
               \int_{\R^d} \lambda(B_{v(t_1-s)}(0)\cap B_{v(t_2-s)}(x))\diff x \nonumber\\
             &=\bvol^2 \int_0^\infty v^{2d}\nu(\diff v) \int_0^{t_1\wedge t_2}
               (t_1-s)^d(t_2-s)^d\theta(\diff s) \nonumber\\
             &=\bvol^2\spm[2d]  \int_0^{t_1\wedge t_2}
               (t_1-s)^d(t_2-s)^d\theta(\diff s)=: \ell(t_1,t_2)
\end{align}
which is symmetric in $t_1$ and $t_2$. Thus, changing the order of the
integrals in the second step and, using \eqref{eq:Gbd} for the final
step, we get
\begin{align}\label{eq:f1}
    \int_{\XX} &f_\alpha^{(1)}(\by)^2\mu(\diff \by)
                 =\int_{\XX}\int_{\by\preceq \bx_1}\int_{\by\preceq \bx_2} G(\bx_1)w(t_{x_1})^\alpha G(\bx_2)w(t_{x_2})^\alpha \mu(\diff \bx_1)\mu(\diff \bx_2)\mu(\diff \by)\nonumber\\
               &=\int_\XX\int_\XX\left(\int_{\by\preceq \bx_1,\by\preceq\bx_2}
                 \mu(\diff \by)\right)G(\bx_1)G(\bx_2)
                 \big(w(t_{x_1})w(t_{x_2}) \big)^\alpha
                 \mu(\diff \bx_1)\mu(\diff \bx_2)\nonumber\\
               &=\int_\XX\int_\XX
                 \mu(L_{x_1,t_{x_1}}\cap L_{x_2,t_{x_2}})
                 G(\bx_1)G(\bx_2)
                 \big (w(t_{x_1})w(t_{x_2}) \big)^\alpha
                 \mu(\diff \bx_1)\mu(\diff \bx_2)\nonumber\\
               &\leq 36 \bvol^{10} \, M_0 I_2,
\end{align}
where
\begin{multline*}
  I_2:=\int_{\R_+^3}\left(\int_{\R^d} \left(\int_{\R^d}
      \mu(L_{0,t_{x_1}}\cap L_{x_2-x_1,t_{x_2}}) \diff x_2 \right)
    \ind_{x_1 \in W + B_{v_{x_1}a}(0)}  \diff x_1 \right)\\
  \times p(v_{x_1}) \big (w(t_{x_1})w(t_{x_2}) \big)^\alpha
  \theta^2(\diff (t_{x_1}, t_{x_2}) \nu (\diff v_{x_1}).
\end{multline*}
By \eqref{eq:Steiner} and \eqref{eq:muint}, we have
\begin{align*}
  	I_2&=\int_{\R_+^3} \ell(t_{x_1},t_{x_2})
  	\lambda(W + B_{v_{x_1}a}(0)) p(v_{x_1}) \big (w(t_{x_1})w(t_{x_2}) \big)^\alpha
  	 \theta^2(\diff (t_{x_1}, t_{x_2}) \nu (\diff v_{x_1})\\
  	& \le c_d (1+a^d) M \,  V(W) \,\int_{\R_+^2} \ell(t_{x_1},t_{x_2}) \big (w(t_{x_1})w(t_{x_2}) \big)^\alpha \theta^2(\diff (t_{x_1}, t_{x_2})).
\end{align*}
Using that $w$ is a decreasing function, the result now follows from
\eqref{eq:f1} and \eqref{eq:muint} by noticing that
\begin{align*}
    &\int_{\R_+^2} \ell(t_{x_1},t_{x_2}) \big (w(t_{x_1})w(t_{x_2}) \big)^\alpha \theta^2(\diff (t_{x_1}, t_{x_2})) \\
    &= \bvol^{2} \spm[2d]
      \int_{\R_+^2}
      \int_0^{t_{x_1}\wedge t_{x_2}}
      (t_{x_1}-s)^d(t_{x_2}-s)^d
      \big (w(t_{x_1})w(t_{x_2}) \big)^\alpha 
      \theta(\diff s) \theta^2(\diff (t_{x_1}, t_{x_2}))\\
    &=\bvol^{2} \spm[2d]
      \int_0^\infty \left(\int_s^\infty
      (t-s)^d
      w(t)^\alpha
      \theta(\diff t)\right)^2\theta(\diff s) \\
     & \le \bvol^{2} \spm[2d]
      \int_0^\infty \left(\int_0^\infty
     t^d
     w(t)^{\alpha/2}
     \theta(\diff t)\right)^2 w(s)^\alpha\theta(\diff s) = \bvol^{2}
       \spm[2d] Q(d,\alpha\spm/2)^2 Q(0,\alpha\spm). \qedhere
\end{align*}
\end{proof}

Arguing as in \eqref{eq:Steiner}, we also have
\begin{align}
	\int_{\R_+}\lambda(W+B_{v_xa}(0)) v_x^d p(v_x) \nu(\diff v_x) &\le \omega (1+a^d)  \sum_{i=0}^d V_{d-i}(W) M_{d+i}\label{eq:Stimp'}\\
	&\le c_d (1+a^d) M' \,  V(W),\label{eq:Steiner'}
\end{align}
with
$$
M':=\int_0^\infty (1+v^d) v^d p(v) \nu(\diff v)
=\spm+\spm[2d] + Q(d,\zeta\spm)^5 (\spm[6d]+\spm[7d]).
$$
Note that by positive association, we have $\spm M \le M'$.
\begin{lemma}\label{lem:2}
  For $a \in (0,\infty)$, $\alpha>0$ and $f_\alpha^{(2)}$ defined at \eqref{eq:fa1},
  $$
  \int_{\XX} f_\alpha^{(2)}(\by)\mu(\diff \by) \le 
  C_1 \,V(W)\; \text{ and }\;
  \int_{\XX} f_\alpha^{(2)}(\by)^2\mu(\diff \by) \le
  C_2 \,V(W)
  $$
  for 
  \begin{align*}
    C_1&:=C(1+a^d) M' \, 
    Q(0,\alpha\spm/2)Q(d,\alpha\spm/2)\,,\\
    C_2&:= C (1+a^d) \,M_d M'\,
    Q(0,\alpha\spm/3)^2 Q(2d,\alpha\spm/3),
  \end{align*}
for a constant $C \in (0,\infty)$ depending only on $d$.
\end{lemma}
\begin{proof}
	By the definition of $f_\alpha^{(2)}$, \eqref{eq:Gbd} and \eqref{eq:Steiner'}, we obtain
	\begin{align*} 
	&\int_{\XX} f_\alpha^{(2)}(\by) \mu(\diff \by)\\
	&\leq 6 \bvol^5  \, 
	\int_\XX\left(\int_{\bx\preceq \by} w(t_y)^{\alpha} \mu(\diff\by)\right)
	 \ind_{x \in W+B_{v_xa}(0)} p(v_x) \mu(\diff \bx)
	\\
	&=6 \bvol^6
	 \int_0^\infty \int_{t_x}^\infty
	w(t_y)^\alpha (t_y-t_x)^d \int_0^\infty \lambda(W+B_{v_xa}(0)) v_x^d p(v_x)  \nu(\diff v_x) \theta(\diff t_y)\theta(\diff t_x)	
	\\
	& \le 6 \bvol^6 c_d (1+a^d) M' \,  V(W)
	\int_0^\infty w(t_x)^{\alpha/2}  \theta(\diff t_x)
	\int_{0}^\infty t_y^d w(t_y)^{\alpha/2} \theta(\diff t_y)
	\\
	&\le 6 \bvol^6 c_d (1+a^d) M' \, 
	Q(0,\alpha\spm/2)Q(d,\alpha\spm/2)\, V(W),
	\end{align*}
	where in the penultimate step, we have used  that $w$ is decreasing.
	This proves the first  assertion.
	
	For the second assertion, using \eqref{eq:Gbd}, we have
	\begin{align}\label{eq:f2}
	\int_\XX f_\alpha^{(2)}(\by)^2&\mu(\diff \by)
	=\int_\XX w(t_y)^{2\alpha} \left(
	\int_{\bx_1\preceq \by}G(\bx_1)\mu(\diff \bx_1)
	\int_{\bx_2\preceq \by}G(\bx_2)\mu(\diff \bx_2)
	\right)\mu(\diff \by)\nonumber\\
	&=\int_\XX\int_\XX\left(
	\int_{\bx_1\preceq \by,\bx_2\preceq\by}
	w(t_y)^{2\alpha}
	\mu(\diff\by)\right)G(\bx_1)G(\bx_2)
	\mu(\diff \bx_1)\mu(\diff \bx_2)\nonumber\\
	&\leq 36 \bvol^{10} \,\int_\XX\int_\XX
	p(v_{x_1})p(v_{x_2})
	\nonumber\\
	&
	\times\ind_{x_1 \in W+B_{v_{x_1}a}(0)}
	\left(\int_{\bx_1\preceq \by,\bx_2\preceq\by}
	w(t_y)^{2\alpha}
	\mu(\diff \by)\right)\mu(\diff \bx_1)\mu(\diff \bx_2).
	\end{align}
	For fixed $\bx_1,t_{x_2}$ and $v_{x_2}$, we have
	\begin{align*}
	&\int_{\R^d}\int_{\bx_1\preceq \by,\bx_2\preceq\by} 
	w(t_y)^{2\alpha} \mu(\diff\by)\diff x_2\\
	&=\int_{t_{x_1}\vee t_{x_2}}^\infty
	w(t_y)^{2\alpha}
	\left(\int_{\R^d}\lambda \big (B_{v_{x_1}(t_y-t_{x_1})}(0)
	\cap B_{v_{x_2}(t_y-t_{x_2})}(x) \big)\diff x\right)\theta(\diff t_y)
	\\
	&=\bvol^2 v_{x_1}^d v_{x_2}^d \int_{t_{x_1}\vee t_{x_2}}^\infty
	(t_y-t_{x_1})^d(t_y-t_{x_2})^d w(t_y)^{2\alpha} \theta(\diff t_y)\,.
	\end{align*}
	Arguing similarly as for $\mu(f_\alpha^{(2)})$ above, 
	\eqref{eq:f2} and \eqref{eq:Steiner'} yield that
	\begin{align*}
	\int_\XX &f_\alpha^{(2)}(\by)^2 \mu(\diff \by)
	\leq 36 \bvol^{12} \,M_d
	 \int_0^\infty \lambda \big (W+B_{v_{x_1}a}(0) \big) v_{x_1}^d p(v_{x_1})
	\nu(\diff v_{x_1})\\
	&\qquad\qquad
	 \times\int_{\R_+^2}  \left(
	\int_{t_{x_1}\vee t_{x_2}}^\infty (t_y-t_{x_1})^d(t_y-t_{x_2})^d
	w(t_y)^\alpha \theta(\diff t_y)\right)\theta^2(\diff (t_{x_1},t_{x_2}))
	\\
	&\leq  36 \bvol^{12}  c_d  (1+a^d) \,M_d M' V(W)
	\int_0^\infty w(t_{x_1})^{\alpha/3} \theta(\diff t_{x_1})\\
	& \qquad \qquad \times \int_0^\infty w(t_{x_2})^{\alpha/3} \theta(\diff t_{x_2}) \int_0^\infty t_y^{2d}w(t_y)^{\alpha/3}\theta(\diff t_y) 
	\\
	&\le 36 \bvol^{12}  c_d  (1+a^d) \,M_d M'\,
	Q(0,\alpha\spm/3)^2 Q(2d,\alpha\spm/3)\, V(W). \qedhere
	\end{align*}
\end{proof}

Before proceeding to bound the integrals of $f^{(3)}$, notice that,
since $\theta$ is a non-null measure,
\begin{align}
\label{eq:qtwo}
\qtwo=\qtwo(\spm):&=\int_0^\infty t^{d-1} e^{-\frac{\alpha\spm}{3}\Lambda(t)} \diff t
=\int_0^\infty t^{d-1} e^{-\frac{\alpha\bvol \spm}{3}\int_0^t (t-s)^d\theta(\diff s)} \diff t
\\
&\leq \int_0^\infty t^{d-1} e^{-\frac{\alpha\bvol \spm}{3}\int_0^{t/2} (t/2)^d\theta(\diff s)} \diff t
=\int_0^\infty t^{d-1} e^{-\frac{\alpha\bvol \spm}{3}\,\theta([0,t/2))(t/2)^d} \diff t<\infty\,.\nonumber
\end{align}

\begin{lemma}\label{lem:3}
	For $a \in (0,\infty)$, $\alpha \in (0,1]$ and $f_\alpha^{(3)}$ defined at \eqref{eq:fa1},
	$$
	\int_{\XX} f_\alpha^{(3)}(\by)\mu(\diff \by) \leq
	C_1\,V(W)
	\quad\text{and}\quad
	\int_{\XX} f_\alpha^{(3)}(\by)^2\mu(\diff \by) \leq
	C_2\, V(W),
	$$
	where
	\begin{align*}
	C_1&:=
	C\, (1+a^d) M' Q(0,\alpha\spm/3)^2 
	\Big[\qtwo +
	Q(2d,\alpha\spm/3)
	\spm \Big]\,,
	\\
	C_2&:=
	C\, (1+a^d) M_d M' (1+\spm[2d]\spm^{-2})Q(0,\alpha\spm/3)^3 
	\left(\qtwo^2+\qtwo  Q(2d,\alpha\spm/3) \spm +Q(2d,\alpha\spm/3)^2 \spm[2d]\right),
	\end{align*}
	for a constant $C \in (0,\infty)$ depending only on $d$.
\end{lemma} 
\begin{proof}
  Note that $\bx,\by \preceq \bz$ implies
  $$
  |x-y|\le |x-z|+|y-z| \le t_z(v_x+v_y).
  $$
  For $q$ defined at \eqref{eq:q}, we have
  \begin{align*}
	q(\bx,\by)
	\le e^{-\spm \Lambda(r_0)}
	\int_{r_0}^\infty \lambda \big (B_{v_x(t_z-t_x)}(0) \cap B_{v_y(t_z-t_y)}(y-x) \big)
	e^{-\spm (\Lambda(t_z)-\Lambda(r_0))}  \theta(\diff t_z) \,,
	\end{align*}
	where
	\begin{equation*}
	r_0=r_0(\bx,\by):=\frac{|x-y|}{v_x+v_y}\vee t_x\vee t_y\,.
	\end{equation*}
	Therefore,
	\begin{align}\label{eq:qbd}
	q(\bx,\by)^\alpha
	&\leq e^{-\alpha\spm \Lambda(r_0)}\left(1+
	\int_{r_0}^\infty \lambda \big (B_{v_x(t_z-t_x)}(0) \cap B_{v_y(t_z-t_y)}(y-x) \big)
	e^{-\spm (\Lambda(t_z)-\Lambda(r_0))} \theta(\diff t_z)
	\right)\nonumber
	\\
	&\leq e^{-\alpha\spm \Lambda(r_0)}
	+ \int_{r_0}^\infty \lambda \big (B_{v_x(t_z-t_x)}(0) \cap B_{v_y(t_z-t_y)}(y-x) \big)
	e^{-\alpha\spm \Lambda(t_z)} \theta(\diff t_z)\,.
	\end{align}
	Then, with $f_\alpha^{(3)}$ defined at \eqref{eq:fa1},
	\begin{align}\label{eq:f3spl}
	\int_\XX &f_\alpha^{(3)}(\by)\mu(\by)
	\leq\int_{\XX^2} G(\bx)
	e^{-\alpha\spm \Lambda(r_0)}
	\mu^2(\diff(\bx, \by))\nonumber
	\\
	&\quad +\int_{\XX^2} G(\bx)
	\int_{r_0}^\infty\lambda \big (B_{v_x(t_z-t_x)}(0) \cap B_{v_y(t_z-t_y)}(y-x) \big)
	e^{-\alpha\spm \Lambda(t_z)} \theta(\diff t_z)
	\mu^2(\diff(\bx, \by)).
	\end{align}
	Since $\Lambda$ is increasing, 
	\begin{equation}\label{eq:wbd}
	\exp\{-\alpha\spm \Lambda(r_0(\bx,\by))\}
	\leq \exp\left\{-\frac{\alpha\spm }{3}
	\left[\Lambda\left(\frac{|x-y|}{v_x+v_y}\right)
	+\Lambda(t_x)+\Lambda(t_y)\right]\right\},
	\end{equation}
	and, by a change of variable and passing to polar coordinates, we obtain
	\begin{equation}\label{eq:bd1}
	\int_{\R^d}e^{-\frac{\alpha\spm }{3}
		\Lambda\left(\frac{|x|}{v_x+v_y}\right)}\,\diff x
	\leq d\bvol (v_x+v_y)^d \int_0^\infty \rho^{d-1}
	e^{-\frac{\alpha\spm }{3}\Lambda(\rho)}\diff\rho
	= d\bvol (v_x+v_y)^d \qtwo \,.
	\end{equation}
	Thus, using \eqref{eq:Gbd}, \eqref{eq:wbd} and \eqref{eq:bd1}, we
	can bound the first summand on the right-hand side of \eqref{eq:f3spl}
	as
	\begin{align*}
	\int_{\XX^2} G(\bx)
	&e^{-\alpha\spm \Lambda(r_0)}
	\mu^2(\diff(\bx, \by))
	\leq
	6 \bvol^5
   \int_0^\infty e^{-\frac{\alpha\spm }{3}\Lambda(t_x)}
   \diff t_x
	\int_0^\infty e^{-\frac{\alpha\spm }{3}\Lambda(t_y)} \diff t_y
	\\
	&\quad \times  \int_{\R^d}\ind_{x \in W+B_{v_x a}(0)} \diff x
   \iint_{\R_+^2\times\R^d} p(v_x)
   e^{-\frac{\alpha\spm }{3}\Lambda\left(\frac{|x-y|}{v_x+v_y}\right)}
	\diff y\,\nu^2(\diff (v_x, v_y))
	\\
	&\leq
	6 \bvol^6 d Q(0,\alpha\spm/3)^2 
	\qtwo\,
	\int_{\R_+^2}\lambda \big (W+B_{v_x a}(0) \big) p(v_x)(v_x+v_y)^d 
	\nu^2(\diff (v_x, v_y))
	\\
	&\leq 2^{d+3} \bvol^6 d c_d (1+a^d) M' \qtwo Q(0,\alpha\spm/3)^2 
	 \,V(W)\,,
	\end{align*}
	where for the final step, we have used Jensen's inequality, \eqref{eq:Steiner} and \eqref{eq:Steiner'},
        and that $\spm M \le M'$.  Arguing similarly for
        the second summand in \eqref{eq:f3spl}, using \eqref{eq:Gbd} and that $r_0 \ge t_x \vee t_y$
        in the first, \eqref{eq:intball} in the second, and \eqref{eq:Steiner'} in the final step, we
        obtain
	\begin{align*}
	&\int_{\XX^2} G(\bx)
	\int_{r_0}^\infty\lambda \big (B_{v_x(t_z-t_x)}(0) \cap B_{v_y(t_z-t_y)}(y-x) \big)
	e^{-\alpha\spm \Lambda(t_z)} \theta(\diff t_z)
	\mu^2(\diff(\bx, \by))
	\\
	&\leq 6 \bvol^5\,
	\int_{\R_+^2} \int_{\R_+^2} \lambda \big ( W+B_{v_x a}(0) \big) p(v_x) \int_{t_x\vee t_y}^\infty 
	w(t_z)^\alpha
	\\
	&\qquad \times \left(
	\int_{\R^d}\lambda \big (B_{v_x(t_z-t_x)}(0) \cap B_{v_y(t_z-t_y)}(y) \big)\diff y
	\right) \theta(\diff t_z) \theta^2( \diff (t_x, t_y))\,\nu^2(\diff (v_x, v_y))
	\\
	&\leq 6 \bvol^7 \,
	\int_{\R_+^2} \lambda \big ( W+B_{v_x a}(0) \big) p(v_x) v_x^d v_y^d\,\nu^2(\diff (v_x, v_y))
	\\
	&\qquad\times\int_{\R_+^3} t_z^{2d}
	w(t_z)^{\alpha/3}w(t_x)^{\alpha/3}w(t_y)^{\alpha/3}
	\theta^3(\diff (t_z, t_x, t_y))
	\\
	&\leq
	6 \bvol^7 c_d \, (1+a^d) \, Q(0,\alpha\spm/3)^2 
	Q(2d,\alpha\spm/3)
	\spm M'\,V(W).
	\end{align*}
	This concludes the proof of the first assertion.
	
	Next, we prove the second assertion. For ease of notation, we
        drop obvious subscripts and write $\by=(y,s,v)$,
        $\bx_1=(x_1,t_1,u_1)$ and $\bx_2=(x_2,t_2,u_2)$. Using
        \eqref{eq:qbd}, write
	\begin{align}\label{eq:f32spl}
	\int_\XX f_\alpha^{(3)}(\by)^2\mu(\by)
	&=\int_{\XX^3} G(\bx_1)G(\bx_2)
	q(\bx_1,\by)^\alpha q(\bx_2,\by)^\alpha
	\mu^3(\diff (\bx_1, \bx_2, \by))\nonumber
	\\
	&\leq \int_{\XX^3} G(\bx_1)G(\bx_2)
	(\I_1+2\I_2+\I_3)\mu^3(\diff (\bx_1, \bx_2, \by))\,,
	\end{align}
	with
	\begin{align*}
	\I_1&=\I_1(\bx_1,\bx_2,\by)
	:=\exp\Big\{-\alpha\spm \big[
	\Lambda(r_0(\bx_1,\by))+\Lambda(r_0(\bx_2,\by))
	\big]\Big\},
	\\
	\I_2&=\I_2(\bx_1,\bx_2,\by)
	:=e^{-\alpha\spm \Lambda(r_0(\bx_1,\by))}
	\int_{s\vee t_2}^\infty\lambda(B_{u_2(r-t_2)}(0) 
	\cap B_{v(r-s)}(y-x_2))e^{-\alpha\spm \Lambda(r)} \theta( \diff r),
	\\
	\I_3&=\I_3(\bx_1,\bx_2,\by)
	:=\int_{s\vee t_2}^\infty\lambda \big (B_{u_2(r-t_2)}(0) 
	\cap B_{v(r-s)}(y-x_2) \big)e^{-\alpha\spm \Lambda(r)} \theta(\diff  r)\; 
	\\
	&\qquad\qquad\qquad\qquad\qquad\qquad
	\times\int_{s\vee t_1}^\infty\lambda \big (B_{u_1(\rho-t_1)}(0) 
	\cap B_{v(\rho-s)}(y-x_1) \big)e^{-\alpha\spm \Lambda(\rho)} \theta(\diff \rho)\,.
	\end{align*}
	
	\noindent By \eqref{eq:bd1}, 
	\begin{align*}
	&\iint_{\R^{2d}}\exp\Big\{-\frac{\alpha\spm }{3}\left[
	\Lambda\left(\frac{|y|}{u_1+v}\right)
	+\Lambda\left(\frac{|x-y|}{u_2+v}\right)\right]\Big\}\,\diff x \diff y
	\\
	&\leq\int_{\R^d} \exp\Big\{-\frac{\alpha\spm }{3}
	\Lambda\left(\frac{|y|}{u_1+v}\right)\Big\}\diff y
	\int_{\R^d} \exp\Big\{-\frac{\alpha\spm }{3}
	\Lambda\left(\frac{|x|}{u_2+v}\right)\Big\}\diff x
	\\
	&\leq d^2\bvol^2 (u_1+v)^d (u_2+v)^d
	\qtwo ^2.
	\end{align*}
	Hence, using \eqref{eq:Gbd} and \eqref{eq:wbd} for the first
        step, we
        have
	\begin{align*}
	&\int_{\XX^3} G(\bx_1)G(\bx_2) \I_1(\bx_1,\bx_2,\by)\ \mu^3(\diff (\bx_1,\bx_2,\by))
	\\
	&\leq 36 \bvol^{10}\,
	\int_{\R^{d}} \ind_{x_1 \in W+B_{u_1 a}(0)} \diff  x_1
	\int_{\R_+^3} e^{-\frac{\alpha\spm }{3}
		\left[\Lambda(t_1)+\Lambda(t_2)+2\Lambda(s)\right]}
	\theta^3(\diff (t_1, t_2,s))
	\\
	&\qquad\qquad\times
	\iint_{\R_+^3\times(\R^d)^2}p(u_1) p(u_2)
	e^{-\frac{\alpha\spm}{3}\left[
		\Lambda\left(\frac{|x_1-y|}{u_1+v}\right)
		+\Lambda\left(\frac{|x_2-y|}{u_2+v}\right)\right]}
	\,\diff y\,\diff x_2\,\nu^3(\diff (u_1,u_2,v))
	\\
	&\leq 36 \bvol^{12}  d^2\,\qtwo ^2 
	Q(0,\alpha\spm/3)^{3} \\
	& \qquad\qquad \times \int_{\R_+^3} \lambda \big (W+B_{u_1 a}(0) \big)\,(u_1+v)^d (u_2+v)^d
	p(u_1) p(u_2)
	\nu^3(\diff (u_1,u_2,v))
	\\
	&\leq  c_1 \,\qtwo ^2 
	Q(0,\alpha\spm/3)^{3} (1+a^d)  (1+\spm[2d]\spm^{-2}) M_d M' \, V(W)\,
	\end{align*}
	for some constant $c_1 \in (0,\infty)$ depending only on
        $d$. Here we have used monotonicity of $Q$ with respect to its
        second argument in the penultimate step and, in the final
        step, Jensen's inequality and \eqref{eq:Steiner'} along with the
        fact that
        \begin{equation*} 
        \int_{\R_+^2} (1+\spm^{-1} v^d) (u_2^d+v^d)
        p(u_2)
        \nu^2(\diff (u_2,v)) \le C (1+\spm^{-2}\spm[2d]) M_d
        \end{equation*}
        for some constant $C \in (0,\infty)$ depending only on $d$.
	
	Next, we bound the second summand in \eqref{eq:f32spl}. Using
        \eqref{eq:intball} in the second step, monotonicity of
        $\Lambda$ and \eqref{eq:wbd} in the third step and
        \eqref{eq:bd1} in the final one, we have
	\begin{align*}
	&\iint_{\R^{2d}} \I_2(\bx_1,\bx_2,\by) \diff x_2 \diff y
	\\
	&=\int_{\R^{d}}e^{-\alpha\spm \Lambda(r_0(\bx_1,\by))} \diff y
	\int_{s\vee t_2}^\infty\int_{\R^d}  
	\lambda \big (B_{u_2(r-t_2)}(0) \cap B_{v(r-s)}(y-x_2) \big) \diff x_2\,
	e^{-\alpha\spm \Lambda(r)} \theta(\diff r)
	\\
	&=\bvol^2 u_2^d v^d \int_{\R^{d}}
	e^{-\alpha\spm \Lambda(r_0(\bx_1,\by))} \diff y
	\int_{s\vee t_2}^\infty (r-t_2)^d (r-s)^d\,
	e^{-\alpha\spm \Lambda(r)} \theta(\diff r)
	\\
	&\le \bvol^2 u_2^d v^d \exp\left\{-\frac{\alpha\spm }{3}\left[
	\Lambda(t_1)+\Lambda(t_2)+2\Lambda(s)\right]\right\} 
	\\
	&\qquad \qquad \qquad  \times \int_0^\infty r^{2d}\,
	e^{-\alpha\spm \Lambda(r)/3} \theta(\diff r)
	\int_{\R^d} \exp\Big\{-\frac{\alpha\spm }{3}
	\Lambda\left(\frac{|x_1-y|}{u_1+v}\right)\Big\}\,\diff y
	\\
	&=d \bvol^3 \,\qtwo Q(2d,\alpha\spm/3)u_2^d v^d (u_1+v)^d
	\exp\left\{-\frac{\alpha \spm }{3}\left[
	\Lambda(t_1)+\Lambda(t_2)+2\Lambda(s)\right]\right\}\,.
	\end{align*}
	Therefore, arguing similarly as before, we obtain
	\begin{align*}
	&\int_{\XX^3} G(\bx_1)G(\bx_2) \I_2(\bx_1,\bx_2,\by)
	\ \mu^3(\diff (\bx_1, \bx_2, \by))
	\\
	&\leq
	36 \bvol^{10} \,
	\int_{\R^{d}}\ind_{x_1 \in W+B_{u_1 a}(0)} \diff x_1
	\iint_{\R_+^6} p(u_1)p(u_2) \\
	&  \qquad \qquad\qquad
   \times\left( \iint_{\R^{2d}} \I_2(\bx_1,\bx_2,\by)
   \diff x_2\,\diff y \right)
	\theta^3 (\diff (t_1,t_2,s))\nu^3 (\diff (u_1,u_2,v))
	\\
	&\leq  36 \bvol^{13}  d \,\qtwo Q(2d,\alpha\spm/3)\, \int_{\R_+^3}
	e^{-\frac{\alpha\spm }{3}\left[
		\Lambda(t_1)+\Lambda(t_2)+2\Lambda(s)\right]}
	\theta^3 (\diff( t_1,t_2,s)) 
	\\
	&\qquad\qquad\qquad	\times \int_{\R_+^3} \lambda \big (W+B_{u_1 a}(0) \big) u_2^d
	 v^d (u_1+v)^d p(u_1)p(u_2)
	\nu^3 (\diff (u_1,u_2,v))
	\\
	&\leq c_2 \,	\qtwo Q(2d,\alpha\spm/3)
	Q(0,\alpha\spm/3)^{3} (1+a^d) (1+\spm[2d]\spm^{-2}) \spm M_d M' \, V(W)\,
	\end{align*}
	for some constant $c_2 \in (0,\infty)$ depending only on $d$, where for the final step we have used
	\begin{equation*} 
		\int_{\R_+^2} (1+\spm^{-1} v^d) u_2^d v^d
		p(u_2)
		\nu^2(\diff (u_2,v)) \le C' (1+\spm^{-2}\spm[2d]) \spm M_d
	\end{equation*}
	for some constant $C' \in (0,\infty)$ depending only on $d$.

	Finally, we bound the third summand in \eqref{eq:f32spl}. Arguing as above, 
	\begin{align*}
	&\iint_{\R^{2d}} \I_3(\bx_1,\bx_2,\by)\,\diff x_2\,\diff y
	\\
	&=\int_{s\vee t_1}^\infty
	\left(\int_{\R^d}
	\lambda \big (B_{u_1(\rho-t_1)}(0)\cap B_{v(\rho-s)}(y-x_1) \big) \diff y
	\right)
	e^{-\alpha\spm \Lambda(\rho)} \theta(\diff \rho)
	\\
	&\qquad\qquad \qquad\times\int_{s\vee t_2}^\infty
	\left(\int_{\R^d}
	\lambda \big (B_{u_2(r-t_2)}(0) \cap B_{v(r-s)}(y-x_2) \big) \diff x_2
	\right)
	e^{-\alpha\spm  \Lambda(r)} \theta(\diff r)
	\\
	&=\bvol^4 u_1^d u_2^d v^{2d}
	\int_{s\vee t_1}^\infty
	(\rho-t_1)^d(\rho-s)^d
	e^{-\alpha\spm \Lambda(\rho)} \theta(\diff \rho)
	\int_{s\vee t_2}^\infty
	(r-t_1)^d(r-s)^d
	e^{-\alpha\spm \Lambda(r)} \theta(\diff r)
	\\
	&\leq \bvol^4 u_1^d u_2^d v^{2d} \left(\int_0^\infty r^{2d}\,
	e^{-\alpha\spm \Lambda(r)/3} \theta(\diff r)\right)^2 
	\exp\Big\{
	-\frac{\alpha}{3}\spm \big[
	\Lambda(t_1)+\Lambda(t_2)+2\Lambda(s)
	\big]\Big\}
	\\
	&\le  \bvol^4\,Q(2d,\alpha\spm/3)^2 u_1^d u_2^d v^{2d}
	\exp\Big\{
	-\frac{\alpha}{3} \spm \big[
	\Lambda(t_1)+\Lambda(t_2)+2\Lambda(s)
	\big]\Big\}\,.
	\end{align*}
	Thus, 
	\begin{align*}
	&\int_{\XX^3} G(\bx_1)G(\bx_2) \I_3(\bx_1,\bx_2,\by)
	\ \mu^3(\diff( \bx_1,\bx_2,\by))
	\\
	&\leq 36 \bvol^{10}\,
	\int_{\R^{d}} \ind_{x_1 \in W+B_{u_1 a}(0)}  \diff x_1
	\iint_{\R_+^6} p(u_1) p(u_2)
	\\
	&\qquad\qquad\qquad
	\times\left( \iint_{\R^{2d}} \I_3(\bx_1,\bx_2,\by) \diff x_2\,\diff y \right)
	\theta^3 (\diff (t_1,t_2,s))\nu^3 (\diff (u_1, u_2, v))
	\\
	&\leq 36 \bvol^{14} \,Q(2d,\alpha\spm/3)^2
	\int_{\R_+^3} \lambda \big (W+B_{u_1 a}(0) \big)
	p(u_1) p(u_2)
	u_1^d u_2^d v^{2d}
	\nu^3 (\diff (u_1, u_2, v))
	\\
	&\qquad\qquad\qquad\times
	\int_{\R_+^3}
	e^{-\frac{\alpha\spm }{3}\left[
		\Lambda(t_1)+\Lambda(t_2)+2\Lambda(s)\right]}
	\theta^3 (\diff (t_1,t_2,s)) 
	\\
	&\leq c_3
	\,Q(2d,\alpha\spm/3)^2 Q(0,\alpha\spm/3)^{3} (1+a^d)
	\spm[2d] M_d M'\, V(W)\,,
	\end{align*}
	for some constant $c_3 \in (0,\infty)$ depending only on $d$. Combining the bound for the summands on the right-hand
	side of \eqref{eq:f32spl} yields the desired conclusion.
\end{proof}

To compute the bounds in \eqref{eq:Wass} and \eqref{eq:Kol}, we now
only need to bound $\mu ((\kappa+g)^{2\beta} G)$.

\begin{lemma}
  \label{lem:4}
  For $a \in (0,\infty)$ and $\alpha \in (0,1]$,
  $$
  \mu ((\kappa+g)^{\alpha} G) \le 
  C_1 \, V(W)\,,	
  $$
  where
  \begin{equation*}
    C_1:= C \, (1+a^d) Q(0,\alpha \zeta\spm/2) [M
    +  (M + M')Q(d,\zeta\spm/2)^\alpha]
  \end{equation*}
for a constant $C \in (0,\infty)$ depending only on $d$.
\end{lemma}
\begin{proof}
  Define the function
  \begin{displaymath}
    \psi(t):=\int_t^\infty(s-t)^d
    e^{-\zeta\spm \Lambda(s)}
    \theta(\diff s)\,,
  \end{displaymath}
  so that $g(\bx)=\bvol v_x^d\psi(t_x)$. 
  By subadditivity, it suffices to separately bound
  $$
  \int_\XX \kappa^{\alpha}(\bx)G(\bx)\mu(\diff \bx)
  \quad\text{and}\quad
  \int_\XX g(\bx)^{\alpha}G(\bx)\mu(\diff \bx)\,.
  $$
  By \eqref{eq:Gbd} and \eqref{eq:Steiner},
  \begin{align*}
    \int_\XX \kappa^{\alpha}(\bx)G(\bx)\mu(\diff \bx)
    &\leq 6 \bvol^5 \int_\XX \ind_{x \in W+B_{v_x a}(0)}p(v_x)
      e^{-\alpha\spm \Lambda(t_x)} \,\diff x\,\theta(\diff t_x)\,\nu(\diff v_x)
    \\
    &\le 6\bvol^5 c_d (1+a^d) Q(0,\alpha \zeta\spm/2)
       M\, V(W)\,.
  \end{align*}
	For the second integral, using \eqref{eq:Steiner'} write
	\begin{align*}
	\int g(\bx)^{\alpha}G(\bx)\mu(\diff \bx)
	&\leq 6 \bvol^{5+\alpha}
	\int_0^\infty\int_0^\infty \psi(t_x)^\alpha  \lambda(W+B_{v_x a}(0)) v_x^{\alpha d}
	p(v_x)\nu(\diff v_x)\theta(\diff t_x)
	\\
	&\le 6 \bvol^{6} c_d (1+a^d) (M+M') \, V(W) \,\int_0^\infty \psi(t_x)^\alpha \theta(\diff t_x)
	\end{align*}
	Note that,
	\begin{align*}
	\int_0^\infty\psi(t)^\alpha \theta(\diff t)
	&=\int_0^\infty \left(\int_t^\infty (s-t)^de^{-\zeta\spm\Lambda(s)}\theta(\diff s)\right)^\alpha \theta(\diff t)
	\\
	&\leq
	\int_0^\infty e^{-\alpha \zeta\spm\Lambda(t)/2} \theta(\diff t)
	\left(
	\int_0^\infty s^d e^{-\zeta\spm\Lambda(s)/2}\theta(\diff s)
	\right)^\alpha
	\\
	&=Q(0,\alpha \zeta\spm/2)Q(d,\zeta\spm/2)^\alpha\,,
	\end{align*}
	where we have used the monotonicity of $\Lambda$ in the second step.
	Combining with the above bounds yield the result.
\end{proof}

\begin{proof}[Proofs of Theorems~\ref{thm:birth} and \ref{thm:scale}:]
  Theorem~\ref{thm:birth} follows from \eqref{eq:Wass} and \eqref{eq:Kol} upon
  using Lemmas~\ref{lem:1}, \ref{lem:2}, \ref{lem:3} and \ref{lem:4}
  and including the factors involving the moments of the speed into
  the constants. 

  The upper bound in Theorem~\ref{thm:scale} follows by combining
  Theorem~\ref{thm:birth} and Proposition~\ref{prop:var}, upon
    noting that $V(n^{1/d} W) \le n V(W)$ for $n \in \N$.
	
  The optimality of the bound in Theorem~\ref{thm:scale} in the
  Kolmogorov distance follows by a general argument employed in the
  proof of \cite[Theorem~1.1, Eq.~(1.6)]{En81}, which shows that the
  Kolmogorov distance between any integer-valued random variable,
  suitably normalized, and a standard normal random variable is always
  lower bounded by a universal constant times the inverse of the
  standard deviation, see Section~6 therein for further details. The
  variance upper bound in \eqref{eq:vgam} now yields the result.
\end{proof}

\begin{proof}[Proof of Theorem~\ref{thm:speed}]
	Let $\theta$ be given at \eqref{eq:10}.
	Then, as in the proof of Proposition~\ref{prop:var},
	\begin{align*}
	\Lambda(t)=B\, \bvol t^{d+\tau+1},
	\end{align*}
	where $B:=B(d+1,\tau+1)$.  By \eqref{eq:1a}, for $x \in \R_+$
        and $y>0$,
	\begin{equation}\label{eq:q0}
	Q(x,y)=
	\int_0^\infty t^{x+\tau} e^{-y \bvol B\, t^{d+\tau+1}} \diff t
	=
	\frac{(y\bvol B)^{-\frac{x+\tau+1}{d+\tau+1}}}{d+\tau+1}
	\Gamma\left(\frac{x+\tau+1}{d+\tau+1}\right)
	=C_1 y^{-\frac{x+\tau+1}{d+\tau+1}}
	\end{equation}
for some constant $C_1 \in (0,\infty)$ depending only on $x, \tau$ and $d$.
	Then using the inequality
	$\nu_{\delta} \nu_{7d-\delta} \le \nu_{7d}$ for any $\delta \in [0,7d]$, we have that for any $u \in [0,2d]$,
	$$
		M_u= \int_{\R_+} v^u p(v) \nu(\diff v)=\spm[u]+ Q(d,\zeta\spm)^5 \spm[5d+u]
		\le C_2 \spm[u](1+\spm[5d+u]\spm[u]^{-1}\spm^{-5}) \le C_2 \spm[u](1+\spm[7d]\spm^{-7})
	$$
        for $C_2 \in (0,\infty)$ depending only on $\tau$ and
        $d$, where in the last step, we have used positive
          association and the Cauchy-Schwartz inequality to obtain
$$
\spm[7d] \spm^{-7} \ge \spm[5d+u] \spm[u]^{-1} \spm^{-5} \spm[2d-u] \spm[u] \spm^{-2} \ge \spm[5d+u] \spm[u]^{-1} \spm^{-5}.
$$ 
In particular,
$$
M_0 \le C_2 (1+\spm[7d]\spm^{-7}), \quad 
\text{and} \quad 
M \le C_2 (1+\spm)(1+\spm[7d]\spm^{-7}).
$$
	Similarly, by
        \eqref{eq:qtwo},
	\begin{align*}
	\qtwo:=
	\frac{1}{d+\tau+1}
	\Gamma\left(\frac{d}{d+\tau+1}\right)
	(\alpha B \bvol \spm/3)^{-\frac{d}{d+\tau+1}}
	=C_3 \spm^{-\frac{d}{d+\tau+1}}
	\end{align*}
for some constant $C_3 \in (0,\infty)$ depending only on $\alpha, \tau$ and $d$.
	Also by \eqref{eq:q0}, for $b>0$, 
        $$
        Q(x, by)=b^{-\frac{x+\tau+1}{d+\tau+1}}Q(x,y).
        $$
        Recall the parameters $p=1$, $\beta=1/36$ and
        $\zeta=1/50$. We will need a slightly refined version of 
        Lemmas~\ref{lem:1}--\ref{lem:3} that uses \eqref{eq:Stimp} and \eqref{eq:Stimp'} instead of \eqref{eq:Steiner} and \eqref{eq:Steiner'}, respectively. Arguing exactly as in Lemmas~\ref{lem:1}--\ref{lem:3}, this yields
    \begin{align*}
    \int_{\XX} f_\alpha^{(1)}(\by)\mu(\diff \by) &\le C (1+a^d) \frac{Q(0,\alpha \spm/2)}{\alpha} \sum_{i=0}^d V_{d-i}(W) M_i,\\
     \int_{\XX} f_\alpha^{(2)}(\by)\mu(\diff \by) &\le C(1+a^d) \, 
     Q(0,\alpha\spm/2)Q(d,\alpha\spm/2)  \sum_{i=0}^d V_{d-i}(W) M_{d+i},\\
      \int_{\XX} f_\alpha^{(3)}(\by)\mu(\diff \by) &\le C (1+a^d) Q(0,\alpha\spm/3)^2 
      \Big[\qtwo +
      Q(2d,\alpha\spm/3)
      \spm \Big] \sum_{i=0}^d V_{d-i}(W) M_{d+i},\\
      \mu ((\kappa+g)^{\alpha} G) &\le 
      C (1+a^d) Q(0,\alpha \zeta\spm/2) \\
      & \quad \times \left[\sum_{i=0}^d V_{d-i}(W) M_i+  Q(d,\zeta\spm/2)^\alpha \sum_{i=0}^d V_{d-i}(W) M_{\alpha d+i}\right]\,,
    \end{align*}
and
    \begin{align*}
	\int_{\XX} f_\alpha^{(1)}(\by)^2\mu(\diff \by) &\le C  (1+a^d) \, M_0 \spm[2d] Q(d,\alpha\spm/2)^2 Q(0,\alpha\spm) \sum_{i=0}^d V_{d-i}(W) M_i,\\
	\int_{\XX} f_\alpha^{(2)}(\by)^2\mu(\diff \by) \le &C (1+a^d) \,M_d\,
	Q(0,\alpha\spm/3)^2 Q(2d,\alpha\spm/3) \sum_{i=0}^d V_{d-i}(W) M_{d+i},\\
	\int_{\XX} f_\alpha^{(3)}(\by)^2 \mu(\diff \by) \le &C (1+a^d) M_d (1+\spm[2d]\spm^{-2})Q(0,\alpha\spm/3)^3 \\
	& \quad \times \left(\qtwo^2+\qtwo  Q(2d,\alpha\spm/3)\spm+Q(2d,\alpha\spm/3)^2 \spm[2d]\right) \sum_{i=0}^d V_{d-i}(W) M_{d+i},
\end{align*}
where $C \in (0,\infty)$ is a constant depending only on $d$.
    
     	These modified bounds in combination
        with the above estimates along with the fact that $\spm[i] \le \spm^{-1} \spm[d+i]$ yield that
        there exists a constant $C$ depending only on $d$ and $\tau$
        such that for $i \in \{1,2,3\}$,
        \begin{equation}\label{eq:est1}
        \int_{\XX} f_{2\beta}^{(i)}(\by)\mu(\diff \by) \le C (1+a^d) \spm^{-\frac{\tau+1}{d+\tau+1}-1}(1+\spm[7d]\spm^{-7}) \sum_{i=0}^d V_{d-i}(W) \spm[d+i].
        \end{equation}
       Also, note that by H\"older's inequality and positive association, for $i=0, \ldots, d$, we have 
       $$
       \spm^{-\alpha} \spm[\alpha d+i] \le \spm^{-\alpha} \spm[i]^{1-\alpha} \spm[d+i]^{\alpha} \le \spm^{-1} \spm[d+i].
       $$
       Thus, combining with the estimates above yields that
       there exists a constant $C$ depending only on $d$ and $\tau$ such that
        \begin{equation}\label{eq:est2}
        \mu ((\kappa+g)^{2\beta} G) \le C (1+a^d) \spm^{-\frac{\tau+1}{d+\tau+1}-1}(1+\spm[7d]\spm^{-7}) \sum_{i=0}^d V_{d-i}(W) \spm[d+i].
    \end{equation}
	Arguing similarly, we also obtain that there exists a constant $C$ depending only on $d$ and $\tau$
such that
	\begin{align*}
	\int_{\XX} f_\beta^{(1)}(\by)^2\mu(\diff \by)
	&\le C  (1+a^d) \, \spm^{-\frac{\tau+1}{d+\tau+1}-3} \spm[2d] (1+\spm[7d]\spm^{-7})^2  \sum_{i=0}^d V_{d-i}(W) \spm[d+i]\,,\\
	\int_{\XX} f_\beta^{(2)}(\by)^2\mu(\diff \by) &\le C (1+a^d) \,\spm^{-\frac{\tau+1}{d+\tau+1}-1}
	(1+\spm[7d]\spm^{-7})^2  \sum_{i=0}^d V_{d-i}(W) \spm[d+i]\,,\\
	\int_{\XX} f_\beta^{(3)}(\by)^2\mu(\diff \by) &\le C (1+a^d) \,\spm^{-\frac{\tau+1}{d+\tau+1}-1} (1+\spm[7d]\spm^{-7})^4
	 \sum_{i=0}^d V_{d-i}(W) \spm[d+i].
	\end{align*}
Thus, there exists a constant $C$ depending only on $d$ and $\tau$
such that for $i \in \{1,2,3\}$,
	\begin{equation}\label{eq:est3}
	\int_{\XX} f_\beta^{(i)}(\by)^2\mu(\diff \by) \le C  (1+a^d) \, \spm^{-\frac{\tau+1}{d+\tau+1}-1} (1+\spm[7d]\spm^{-7})^4  \sum_{i=0}^d V_{d-i}(W) \spm[d+i].
\end{equation}
	
        Plugging \eqref{eq:est1}, \eqref{eq:est2} and \eqref{eq:est3} in \eqref{eq:Wass} and
        \eqref{eq:Kol}, and using Proposition~\ref{prop:var} to lower
        bound the variance yield the desired bounds.
\end{proof}

\begin{proof}[Proof of Corollary~\ref{cor:scale2}:]
  Define the Poisson process $\eta^{(s)}$ with intensity measure
  $\mu^{(s)}:=\lambda \otimes\theta\otimes\nu^{(s)}$, where
  $\nu^{(s)}(A):=\nu(s^{-1/d}A)$ for all Borel sets $A$. It is
  straightforward to see that the set of locations of exposed points
  of $\eta_s$ has the same distribution as of those of $\eta^{(s)}$,
  multiplied by $s^{-1/d}$, i.e., the set
  $\{x : \bx \in \eta_s \text{ is exposed}\}$ coincides in
  distribution with $\{s^{-1/d}x : \bx \in \eta^{(s)} \text{ is
    exposed}\}$. Hence, the functional $F(\eta_s)$ has the same
  distribution as $F_{s}(\eta^{(s)})$, where $F_{s}$ is defined as
  in \eqref{eq:F} for the weight function
  $$h(\bx)
  = h_{1,s}(x) h_2(t_x) =\ind_{x \in W_s} \ind_{t_x <a}\,,$$
 where $W_s:=s^{1/d} W$. It is easy
  to check that for $k \in \N$, the $k$-th moment of $\nu^{(s)}$ is
  given by $\nu^{(s)}_k=s^{k/d}\nu_k$ and $\lambda(W_s)=s \lambda(W)$. We also have
  $$V_{\nu^{(s)}}(W_s)=\sum_{k=0}^d V_{d-i}(s^{1/d} W) \nu_{d+i}^{(s)} = \sum_{k=0}^d s^{\frac{d-i}{d}} V_{d-i}(W) s^{\frac{d+i}{d}}\nu_{d+i }=s^2 V_\nu (W).
  $$
  Finally noticing that
$$
l_{a,\tau}(\spm^{(s)})=\gamma\left(\frac{\tau+1}{d+\tau+1},a^{d+\tau+1} s\spm \right) (s\spm)^{-\frac{\tau+1}{d+\tau+1}} \ge \gamma\left(\frac{\tau+1}{d+\tau+1},a^{d+\tau+1} \spm \right) (s\spm)^{-\frac{\tau+1}{d+\tau+1}} 
$$
for $s\geq 1$, the result follows directly from
Theorem~\ref{thm:speed}. The optimality of the Kolmogorov bound
follows arguing as in the proof of Theorem~\ref{thm:scale}.
\end{proof}


\section*{Acknowledgment}

The authors would like to thank Matthias Schulte for raising the idea
of extending the central limit theorem to functionals of the
birth-growth model with random growth speed. They are grateful for the
referees for pointing out connections to other papers and encouraging
exploring an applied motivation for our model. IM and RT have been
supported by the Swiss National Science Foundation Grant
No. 200021\_175584. A major part of the work was done when CB was employed by the University of Luxembourg.

\bibliography{BMT}
\bibliographystyle{abbrv}
    
\end{document}